\documentclass[12pt,reqno]{amsart}
\usepackage{fullpage}
\usepackage{amssymb,amsmath}
\usepackage{stmaryrd}
\usepackage{fullpage}
\usepackage{amscd}
\usepackage[all]{xy}
\usepackage{comment}
\usepackage{mathrsfs}
\usepackage{verbatim}
\usepackage{longtable}
\usepackage{enumitem}

\newcommand{\unr}{\operatorname{unr}}

\usepackage{tikz}
\usetikzlibrary{positioning, calc}
\usepackage[pagewise]{lineno}

\usepackage{color}

\usepackage[backref,pdfauthor={Taylor Dupuy}]{hyperref}
\newcommand{\C}{{\mathbb C}}

\newcommand{\PP}{{\mathbb P}}
\newcommand{\Q}{{\mathbb Q}}

\newcommand{\Z}{{\mathbb Z}}

\newcommand{\Zbar}{{\overline{\Z}}}

\newcommand{\Kbar}{{\overline{K}}}

\newcommand{\Fbar}{{\overline{F}}}

\newcommand\td\widetilde

\newcommand{\FF}{{\mathcal F}}

\newcommand{\LL}{{\mathcal L}}
\newcommand{\OO}{{\mathcal O}}

\def\presuper#1#2%
{\mathop{}%
\mathopen{\vphantom{#2}}^{#1}%
\kern-\scriptspace%
#2}

\def\Q{\mathbb{Q}}
\def\C{\mathbb{C}}

\def\P{\mathbb{P}}

\def\Z{\mathbb{Z}}

\def\<{\ensuremath{\langle}}
\def\>{\ensuremath{\rangle}}

 \DeclareMathOperator{\im}{im}

 \DeclareMathOperator{\Hom}{Hom}

 \DeclareMathOperator{\Aut}{Aut}

\DeclareMathOperator{\ord}{ord} 
\DeclareMathOperator{\Div}{Div} \DeclareMathOperator{\Pic}{Pic}

\DeclareMathOperator{\Spec}{Spec}

\newcommand{\GL}{\operatorname{GL}}
\newcommand{\SL}{\operatorname{SL}}

\newcommand\djunion\amalg

\numberwithin{equation}{section}
\newtheorem{theorem}{Theorem}[subsection]
\newtheorem{lemma}[theorem]{Lemma}

\newtheorem{claim}[theorem]{Claim}

\theoremstyle{definition}
\newtheorem{definition}[theorem]{Definition}

\theoremstyle{remark}
\newtheorem{remark}[theorem]{Remark}

\newcommand{\QQ}{\mathbb{Q}}
\newcommand{\ZZ}{\mathbb{Z}}
\newcommand{\RR}{\mathbb{R}}
\renewcommand{\FF}{\mathbb{F}}
\newcommand{\NN}{\mathbb{N}}

\newcommand{\Diff}{\operatorname{Diff}}

\newcommand{\Disc}{\operatorname{Disc}}

\newcommand{\diff}{\operatorname{diff}}

\newcommand{\QQbar}{\overline{\mathbb{Q}}}
\newcommand{\CC}{\mathbb{C}}
\newcommand{\ann}{\operatorname{ann}}

\newcommand{\hull}{\operatorname{hull}}

\newcommand{\Ical}{\mathcal{I}}
\newcommand{\vu}{\underline{v}}
\newcommand{\Vu}{\underline{V}}

\newcommand{\Cond}{\operatorname{Cond}}
\newcommand{\ram}{\operatorname{ram}}
\newcommand{\indone}{\operatorname{ind1}}
\newcommand{\indtwo}{\operatorname{ind2}}
\newcommand{\indthree}{\operatorname{ind3}}
\newcommand{\quu}{\underline{\underline{q}}}

\newcommand{\E}{\mathbb{E}}

\newcommand{\LGP}{\operatorname{LGP}}

\newcommand{\lgp}{\operatorname{lgp}}

\newcommand{\nlogn}{\overline{\log\nu}}

\newcommand{\deghat}{\widehat{\deg}}
\newcommand{\deghatu}{\underline{\widehat{\deg}}}
\newcommand{\Divhat}{\widehat{\operatorname{Div}}}

\newcommand{\bad}{\operatorname{bad}}

\newcommand{\Stab}{\operatorname{Stab}}

\makeatletter
\newcommand*{\da@rightarrow}{\mathchar"0\hexnumber@\symAMSa 4B }
\newcommand*{\da@leftarrow}{\mathchar"0\hexnumber@\symAMSa 4C }
\newcommand*{\xdashrightarrow}[2][]{%
	\mathrel{%
		\mathpalette{\da@xarrow{#1}{#2}{}\da@rightarrow{\,}{}}{}%
	}%
}
\newcommand{\xdashleftarrow}[2][]{%
	\mathrel{%
		\mathpalette{\da@xarrow{#1}{#2}\da@leftarrow{}{}{\,}}{}%
	}%
}
\newcommand*{\da@xarrow}[7]{%
	\sbox0{$\ifx#7\scriptstyle\scriptscriptstyle\else\scriptstyle\fi#5#1#6\m@th$}%
	\sbox2{$\ifx#7\scriptstyle\scriptscriptstyle\else\scriptstyle\fi#5#2#6\m@th$}%
	\sbox4{$#7\dabar@\m@th$}%
	\dimen@=\wd0 %
	\ifdim\wd2 >\dimen@
	\dimen@=\wd2 %   
	\fi
	\count@=2 %
	\def\da@bars{\dabar@\dabar@}%
	\@whiledim\count@\wd4<\dimen@\do{%
		\advance\count@\@ne
		\expandafter\def\expandafter\da@bars\expandafter{%
			\da@bars
			\dabar@ 
		}%
	}%  
	\mathrel{#3}%
	\mathrel{%   
		\mathop{\da@bars}\limits
		\ifx\\#1\\%
		\else
		_{\copy0}%
		\fi
		\ifx\\#2\\%
		\else
		^{\copy2}%
		\fi
	}%   
	\mathrel{#4}%
}
\makeatother

\newcommand\subsetsim{\mathrel{%
			\ooalign{\raise0.2ex\hbox{$\subset$}\cr\hidewidth\raise-0.8ex\hbox{\scalebox{0.9}{$\sim$}}\hidewidth\cr}}}
		
\newcommand{\logmubar}{\overline{\ln \mu}}
\newcommand{\rad}{\operatorname{rad}}
\newcommand{\epsu}{\underline{\epsilon}}
\renewcommand{\mod}{\operatorname{mod}}
\newcommand{\Mu}{\underline{M}}
\newcommand{\Peel}{\operatorname{Peel}}
\renewcommand{\LL}{\mathbb{L}}
\newcommand{\id}{\operatorname{id}}
\newcommand{\Int}{\operatorname{Int}}
\newcommand{\OObar}{\overline{\mathcal{O}}}
\newcommand{\Tbar}{\overline{T}}
\newcommand{\betabar}{\overline{\beta}}
\newcommand{\Lbar}{\overline{L}}
\newcommand{\vbar}{\overline{v}}
\newcommand{\ch}{\operatorname{char}}

\renewcommand{\AA}{\mathbb{A}}

\newcommand{\EE}{\mathbb{E}}
\newcommand{\Irm}{\operatorname{I}}
\newcommand{\IIrm}{\operatorname{II}}
\newcommand{\IIIrm}{\operatorname{III}}
\newcommand{\IVrm}{\operatorname{IV}}
\newcommand{\Vrm}{\operatorname{V}}

\newcommand{\Diffbar}{\overline{\operatorname{Diff}}}
\newcommand{\diffbar}{\overline{\diff}}
\newcommand{\ebar}{\overline{e}}
\newcommand{\Ubar}{\overline{U}}

\begin{document}
%\pagewiselinenumbers

%%%%%%%%%%%%%%%%%%%%%%%%%%%%%%%%%%%%%%%%%%%
%%%%%%%%%%%%%% TITLE 
%%%%%%%%%%%%%%%%%%%%%%%%%%%%%%%%%%%%%%%%%%%
\title{Probabilistic Szpiro, Baby Szpiro, and Explicit Szpiro from Mochizuki's Corollary 3.12 %\ \\ \ \\
%\emph{DRAFT} \ \\
} 
\author{Taylor Dupuy and Anton Hilado}

%Fill in the crap below with your information and delete this comment. 

\date{\today}
\thanks{}

\begin{abstract}
In \cite{Dupuy2020a} we gave some explicit formulas for the ``indeterminacies'' $\indone,\indtwo,\indthree$ in Mochizuki's Inequality as well as a new presentation of initial theta data.
In the present paper we use these explicit formulas, together with our probabilistic formulation of \cite[Corollary 3.12]{IUT3} to derive variants of Szpiro's inequality (in the spirit of \cite{IUT4}). 
In particular, for an elliptic curve in initial theta data we show how to derive uniform Szpiro (with explicit numerical constants). 
The inequalities we get will be strictly weaker than \cite[Theorem 1.10]{IUT4} but the proofs are more transparent, modifiable, and user friendly.
All of these inequalities are derived from an probabilistic version of \cite[Corollary 3.12]{IUT3} formulated in \cite{Dupuy2020a} based on the notion of random measurable sets. 
\end{abstract}

\maketitle
\setcounter{tocdepth}{1}
\tableofcontents

\section{Introduction}\label{S:introduction}

In \cite{Dupuy2020a} we gave a probabilistic interpretation of Mochizuki's Inequality (Corollary 3.12 of \cite{IUT3}). 
In the present paper we perform some explicit computations using this inequality to derive three inequalities which we will call ``Probabilistic Szpiro'', ``Baby Szpiro'', and ``Explicit Szpiro''.
All of these inequalities depend on hypothesis of an elliptic curve being in ``initial theta data built from the field of moduli'', \cite[Corollary 3.12]{IUT3}, and some assumed behavior at the archimedean place stated in Claim~\ref{L:arch-log-bounds}.

In order to state our results we need to talk about initial theta data.
As formulated in \S5 of \cite{Dupuy2020a}, initial theta data is a tuple 
 $$(\Fbar/F, E_F, l, \Mu, \Vu, V^{\bad}_{\mod}, \epsu)$$ 
surrounding an elliptic curve $E=E_F$ over a field $F$ satisfying various conditions which are not important for the purposes of the introduction (the curious reader should consult \cite[\S5]{Dupuy2020a}). 
What is important are the data types of the tuple: the entry $l$ is a fixed prime and, in the present paper, the choices of $\Mu$, and $\epsu$ will be irrelevant. 
We will discuss the sets of places $\Vu$ and $V^{\bad}_{\mod}$ momentarily.
This requires some set up. 

In order to define the sets of places $V^{\bad}_{\mod}$ and $\Vu$ we need to introduce the fields $F_0$ and $K$ to which they belong. 
The field $F_0$ is the \emph{field of Moduli} of the elliptic curve defined by 
$$F_0 := \QQ(j_E),$$ 
in Mochizuki's notation this is $F_{\mod}$.  
The field $K$ is the $l$-division field of $F$ given by
\begin{equation}\label{eqn:defn-of-k}
K := F(E[l]),
\end{equation}
obtained by adjoining the $l$-torsion of $E(\Fbar)$ to $F$.
In this paper, for any field $L$ we will let $V(L)$ denote the collection of places of $L$ and for any non-archimedean place $v \in V(L)$ we will let $\kappa(v)$ denote the residue field and $L_v$ denote the completion of $L$ at $v$.

We now come to the definitions of $V^{\bad}_{\mod}$ and $\Vu$. 
First $V^{\bad}_{\mod}\subset V(F_0)$ is a non-empty set of bad multiplicative places over the field of moduli: for every $E_0$ an elliptic curve over $F_0$ such that $E \cong E_0\otimes_{F_0}F$, if $v\in V^{\bad}_{\mod}$, then $E_0$ has multiplicative reduction at $v$. 
Next, the set $\Vu \subset V(K)$ is a set that maps bijectively to $V(F_0)$ under the natural map $V(K) \to V(F_0)$.

We will be using these quantities momentarily but first we need to describe a special type of initial theta data that will be used in the course of this manuscript. 
For computational purposes one can always take an elliptic curve over its field of moduli (satisfying some mild conditions) and base change this field to a larger field to obtain some curve that can be put in initial theta data. 
We call such theta data ``built from the field of moduli''. 
The precise definition is below.  
\begin{definition}
\label{S:initial-theta-data} Let $E/F$ be an elliptic curve inside initial theta data 
$$(\Fbar/F, E_F, l, \Mu, \Vu, V^{\bad}_{\mod}, \epsu).$$ 
We will say that such a tuple is \emph{built from the field of moduli} provided $E=E_0 \otimes_{F_0} F$ where $F_0=\QQ(j_E)$ is the field of moduli of $E$, $E_0$ is a model of $E$ over $F_0$, $F := F_0(\sqrt{-1},E_0[30]),$
and $V^{\bad}_{\mod} \subset V(F_0)$ is the full set of places of multiplicative reduction.\footnote{The definition of initial theta data in \S5 of \cite{Dupuy2020a} precludes bad places from having residue characteristic two. Also the set of bad places needs to be nonempty. }
\end{definition}
%In what follows we make use of the field extension 
%$$ K := F(E[l]) \supset F, $$
%where $l$ is the prime appearing in initial theta data. 
%We in later sections we make use of the extensions
%$$ K := F(E[l]) \supset F := F_2(E_0[15]) =\supset F_2' := F_0(\sqrt{-1}, E_0[2]) \supset F_0 \supset \QQ $$
%for the auxillary prime $l$ in initial theta data.
%Note that in Mochizuki's notation $F_0 = F_{\mod}$.
We will often use the notation
$$ d_0 := [F_0:\QQ]. $$
The constants in our Szpiro-like inequalities will depend on this degree.
%Also, note that in this document we are building up the initial theta data as in \cite{IUT4} that is 
%we start with and elliptic curve $E_0$ defined over a field of moduli then form the field $F = F_0(\sqrt{-1}, E_0[30l])$ then assume that we have a tuple 
In stating our results we recall from \cite{Dupuy2020a} that for a rational prime $p$ that $V(F_0)_p$ is given the structure of a probability space where $\Pr: V(F_0)_p \to [0,1]$ is defined by 
 $$ \Pr(v):= \frac{[F_{0,v}: \QQ_p]}{[F_0:\QQ_p]}.$$
Using $\Vu$ one can define some interesting probabilistic quantities which appear in our Probabilistic Szpiro inequality and give a good sense of the types of things that Mochizuki's inequality ``knows about''.
\begin{definition}
\begin{enumerate}
	\item The \emph{probability that $\vu \in \Vu_p$ is unramified} is 
	\begin{equation}
	\PP_{\unr,p} = \sum_{w\in \lbrace v \in V(F_0)_p : e(\vu/p)=1 \rbrace } \Pr(w). 
	\end{equation}
	\item The \emph{average ramification degree} of $\vu \in \Vu_p$ is defined to be 
	\begin{equation}
	\ebar_p = \EE( e(\vu/p)). 
	\end{equation}
	\item The \emph{average different of $\Vu/\QQ$} is defined to be
	\begin{equation}\label{E:avg-diff}
	\Diffbar(\Vu/\QQ) = \prod_p p^{\diffbar_p}.
	\end{equation}
	In \eqref{E:avg-diff} we have $\diffbar_p = \log_p( \EE(p^{\diff(\vu/p)}))$ and $\diff(\vu/p) = \ord_p( \Diff(K_{\vu}/\QQ_p))$; $\Diff(K_{\vu}/\QQ_p)$ is the different of $K_{\vu}$ over $\QQ_p$.
\end{enumerate}
\end{definition}

Using these quantities we can now state the Probabilistic Szpiro.
\begin{theorem}[Probabilistic Szpiro]\label{T:probabilistic-szpiro}
Assume \cite[Corollary 3.12]{IUT3} and Claim \ref{L:arch-log-bounds}. For any elliptic curve $E/F$ in initial theta data $(\Fbar/F, E_F, l, \Mu, \Vu, V^{\bad}_{\mod}, \epsu)$ built from the field of moduli we have 
	\begin{equation}\label{E:prob-szpiro}
	\frac{1}{6+\varepsilon_l}\frac{\ln \vert \Delta^{\min}_{E/F} \vert }{[F:\QQ]} \leq \ln \overline{\Diff}(\Vu/\QQ) + \sum_{p} \ln(\overline{e}_p) + A_{l,\Vu}
	\end{equation}
	where 
	 $$A_{l,\Vu}= \ln(\pi) + \sum_p (1 - \PP_{\unr,p}^{l+1/2}) \left (\ln(b_p) + \frac{5}{l+4} \right ) ,$$
	and $b_p = 1/\exp(1)\ln(p)$, and $\varepsilon_l = 24(l+3)/(l^2+l-12)$.
	%Here $\pi = 3.14159265\ldots$.
\end{theorem}
\iffalse 
\begin{remark}
An interesting part about this ``Probabilistic Szpiro'' inequality is the term $A_{l,\Vu}$. 
Its nonconstant part is made up of two terms which are roughly proportional 
$\# \lbrace p : \exists \vu\in\Vu_p \mbox{ such that } e(\vu/p)>1 \rbrace $. 
The interesting part is that the main contribution of this is negative as $b_p<1$ for $p>2$.
That means this term sort-of takes away from the logarithmic ramification indices. 	
%Also, note that because $K$ is Galois over $F_0$ the choice of $\Vu$ doesn't matter so much in the definition of these invariants.
\end{remark}
\fi
From the Probabilistic Szpiro Inequality we can derive a ``Baby Szpiro'' Inequality. 
This inequality only depends on discriminant and degree of the division field $K$.
This inequality can be derived quickly dispensing with a discussion of ramification of the mod $l$ Galois representation and its relation to the conductor (which is reviewed in \S \ref{S:geometry-and-galois}).
In what follows for an ideal $I$ in a ring of integers $R$ we let $\vert I \vert$ denote the absolute norm. 
\begin{theorem}[Baby Szpiro]
	Assume \cite[Corollary 3.12]{IUT3} and Claim \ref{L:arch-log-bounds}.
	Then for any elliptic curve $E/F$ in initial theta data $(\Fbar/F, E_F, l, \Mu, \Vu, V^{\bad}_{\mod}, \epsu)$ built from the field of moduli we have  
	\begin{equation}
		\frac{1}{6+\varepsilon_l} \frac{\ln \vert \Delta^{\min}_{E/F}\vert }{[F:\Q]} \leq \ln( [K:\Q]^{5/4}) \ln( \vert \Disc(K/\Q)\vert^{5/4}) + \ln(\pi).
	\end{equation}
 	In the above formula $\varepsilon_l = (24l+72)/(l^2+l-12)$ and $K=\QQ(j_E,E[30l],\sqrt{-1})$.
\end{theorem}

\iffalse 
\begin{remark}
 The techniques in the proof perhaps indicate that the $5/4$th exponent can be reduced.
 It is unclear to the authors what exponents work in this inequality.
 In \S\ref{S:baby} we pose some questions along these lines. 
\end{remark}
\fi 

In \S\ref{S:explicit} we are more careful with our field invariants and delicately apply the theory of \cite{Serre1968}.
The result of this delicate work is our ``Explicit Szpiro".

\begin{theorem}[Explicit Szpiro]
Assume \cite[Corollary 3.12]{IUT3} and Claim~\ref{L:arch-log-bounds}.
If $E/F$ is an elliptic curve in initial theta data  $(\Fbar/F, E_F, l, \Mu, \Vu, V^{\bad}_{\mod}, \epsu)$ built from the field of moduli then 
\begin{equation}\label{E:explicit-szpiro}
 \vert \Delta^{\min}_{E/F}\vert \leq e^{A_0d_0^2l^4 + B_0d_0}( \vert \Cond(E/F) \vert \cdot \vert \Disc(F/\QQ) \vert)^{24+\varepsilon_l},
\end{equation}
where 
\begin{align*}
A_0 &= 84372107405,\\
B_0 &= 316495,\\
\varepsilon_l &= 96 \left(l + 3\right)/(l^{2} + l - 12).
\end{align*}
\end{theorem}

All of these computations follow the same pattern. 
First, one computes an upper bound for the so-called hull of the multiradial representation (c.f. \cite[\S4]{Dupuy2020a})
then one tries to be as clever as possible in order to obtain futher bounds on the radiuses of these bounding polydiscs.
The initial bounding is performed in \S\ref{S:hulls} and the secondary bounds come in subsequent sections. 
All of the material before $\S 7$ is built for this imminent application. 
A lot of interesting Mathematics comes into play at both stages of these computations and interestingly, as our treatment will show, it seems that there is a lot 
of room for improvement. 
We have put in a great deal of effort in an attempt to highlight some of these avenues for improvement which we hope readers will take an interest in (see for example Remark~\ref{R:ways-to-improve}).

%%%%%%%%%%%%%%%%%%%%%%%%%%%%%%%%%%%%%%%%%%%%%%%%%%%%%%%%
\subsection*{Acknowledgements}
%%%%%%%%%%%%%%%%%%%%%%%%%%%%%%%%%%%%%%%%%%%%%%%%%%%%%%%%
This article is very much indebted to many previous expositions of IUT including (but not limited to) \cite{Fesenko2015,Hoshi2018, Kedlaya2015,Hoshi2015,Stix2015,Mok2015,Mochizuki2017,Yamashita2017, Hoshi2017,Tan2018,SS}. 
The first author also greatly benefitted from conversations with many other mathematicians and would especially like to thank Yuichiro Hoshi for helpful discussions regarding Kummer theory and his patience during discussions of the theta link and Mochizuki's comparison; Kirti Joshi for discussions on deformation theory in the context of IUT; Kiran Kedlaya for productive discussions on Frobenioids, tempered fundamental groups, and global aspects of IUT; Emmanuel Lepage for helpful discussions on the p-adic logarithm, initial theta data, aut holomorphic spaces, the log-kummer correspondence, theta functions and their functional equations, tempered fundamental groups, log-structures, cyclotomic synchronization, reconstruction of fundamental groups, reconstruction of decomposition groups, the "multiradial representation of the theta pilot object", the third indeterminacy, the second indeterminacy, discussions on Hodge Theaters, labels, and kappa coric functions, and discussions on local class field theory; Shinichi Mochizuki for his patience in clarifying many aspects of his theory --- these include discussions regarding the relationship between IUT and Hodge Arakelov theory especially the role of "global multiplicative subspaces" in IUT, discussions on technical hypotheses in initial theta data; discussions on Theorem 3.11 and "(abc)-modules", discussions on mono-theta environments and the interior and exterior cyclotomes, discussions of the behavior of various objects with respect to automorphisms and providing comments on treatment of log-links and the use of polyisomorphisms, discussions on indeterminacies and the multiradial representation, discussions of the theta link, discussions on various incarnations of Arakelov Divisors, discussions on cyclotomic synchronization; Chung Pang Mok for productive discussions on the p-adic logarithm, anabelian evaluation, indeterminacies, the theta link, and hodge theaters; Thomas Scanlon for discussions regarding interpretations and infinitary logic as applied to IUT and anabelian geometry.
We apologize if we have forgotten anybody.

The authors also benefitted from the existence of the following workshops: the 2015 Oxford workshop funded by the Clay Mathematics Institute and the EPSRC programme grant \emph{Symmetries and Correspondences}; the 2017 Kyoto \emph{IUT Summit} workshop funded by RIMS and EPSRC; the Vermont workshop in 2017 funded by the NSF DMS-1519977 and \emph{Symmetries and Correspondences} entitled \emph{Kummer Classes and Anabelian Geometry}; the 2018 Vermont Workshop on \emph{Witt Vectors, Deformations and Absolute Geometry} funded by NSF DMS-1801012.

The first author was partially supported by the European Research Council under the European Unions Seventh Framework
Programme (FP7/2007-2013) / ERC Grant agreement no. 291111/ MODAG while working on this project. 

The research discussed in the present paper profited enormously from the generous support of the International Joint Usage/Research Center (iJU/RC) located at Kyoto Universities Research Institute for Mathematical Sciences (RIMS) as well as the Preparatory Center for Research in Next-Generation Geometry located at RIMS.
	  
%%%%%%%%%%%%%%%%%%%%%%%%%%%%%%%%%%%%%%%%%%%%%%%%%%
\section{Explicit Computations in Tensor Packets}\label{S:different}
%%%%%%%%%%%%%%%%%%%%%%%%%%%%%%%%%%%%%%%%%%%%%%%%%%%
The entire purpose of this subsection is Theorem~\ref{T:descent} and the entire purpose of Theorem~\ref{T:descent} is the for hull computation in \S\ref{S:hulls}.
At the end of the day the differents appearing in these sections are what give rise to the conductor term in Szpiro inequalities (in conjunction with the material in \S\ref{S:geometry-and-galois}).

Fix $K_1,\ldots,K_m$ finite extensions of $\QQ_p$. 
Let $L=K_1 \otimes \cdots \otimes K_m$.
We are interested in the difference between the $\Z_p$-lattices\footnote{A \emph{lattice} of a $\QQ_p$-vectors space $V$ is a $\ZZ_p$-submodule $V_0 \subset V$ which is free of rank $\dim(V)$ whose $\QQ_p$-span is all of $V$. } $\OO_{K_1} \otimes \cdots \otimes \OO_{K_m} \subset \OO_L$.  
It turns out that the index of $\OO_{K_1}\otimes \cdots \otimes \OO_{K_m}$ in $\OO_L$ is related to the Differents of $K_i/\Q_p$ which we describe in the subsequent subsections.
Again, this is needed for the hull computations.\footnote{And in explicit cases one can actually proceed differently, say, using conductors in the sense of Commutative Algebra.}

%---------------------------------------------
\subsection{Integral Closures} 
%---------------------------------------------
For a reduced ring $T$, the \emph{total ring of fractions} is defined by $\kappa(T) = S_{\max}^{-1}T$ where $S_{\max}$ is the multiplicatively closed set of non-zero divisors. 
We let $\Int_{\kappa(T)}(T)$ denote the integral closure of $T$ in $\kappa(T)$.

Let $K_1,\ldots,K_m$ be finite extensions of $\QQ_p$ and consider the special case $T=\OO_{K_1} \otimes \cdots\otimes \OO_{K_m}$ (where tensor products are taken over $\ZZ_p$). 
It turns out that $\kappa(T)=K_1\otimes \cdots \otimes K_m$ (where tensor products are over $\QQ_p$).
Let $L = \kappa(T)$. 
We can use the Chinese Remainder Theorem to write $L \cong \bigoplus_{j=1}^r L_j$ where each $L_j$ is a finite extension of $\QQ_p$.  
Define $\OO_L = \bigoplus_{j=1}^r \OO_{L_j}$.
It also turns out that $\Int_L(T) = \OO_L$.

%---------------------------------------------------
\subsection{Differents and Discriminants}\label{S:diff-and-disc}
%---------------------------------------------------
For a discussion on Differents and Discriminants of fields we refer the reader to \cite[III.2]{Neukirch1999} or \cite{Sutherland2015} or \cite{Conrad}. 
A very comprehensive review of Differents in great generality can be found in \cite[0DW4]{stacksproject} and the references therein. 
For $A \supset B$ a finite extension of rings we define the \emph{different ideal} to be 
$$ \Diff(A/B) := \ann_{B}(\Omega_{B/A}). $$
Here $\Omega_{B/A}$ is the module of Kahler differentials.
When $L/L_0$ is an extension of fields we use the notation $\Diff(L/L_0) := \Diff(\OO_L/\OO_{L_0})$. 
For an extension of number fields $L/L_0$ the different ideal and be computed ``as a product'' of local differents (see \cite{Neukirch1999}, for details).
The different also behaves well in towers.
If $K$ is a finite extension of $\QQ_p$ with residue field $k$ then $\OO_K$ can be written as  
$$ \OO_K = W(k)[x]/(f(x)),$$
where $f(x)$ is an Eisenstein polynomial of degree $e$ (also $e$ is the ramification degree of $K/\QQ_p$) and $W(k)$ is the full ring of $p$-typical Witt vectors of $k$.
As $\Diff(W(k)/\ZZ_p) = 1$ (this extension is unramified) to compute $\Diff(\OO_K/\ZZ_p)$ it remains to compute $\Diff(\OO_K/W(k))$. 
From the formula $\Omega_{\OO_K/W(k)} = (\OO_K\cdot dx)/(\OO_K\cdot df)$ and $d(f(x)) = f'(x) dx$ we find that 
\begin{equation}\label{E:diff-formula}
 \Diff(\OO_K/\ZZ_p)=(f'(\pi)).
\end{equation}
We can do some more computations to get some useful information. We find $f'(\pi) = e \pi^{e-1} + \cdots$ where all of terms have distinct valuation and the leading term of $f'(\pi)$ has minimal valuation (this is due to the Eisenstein-ness hypothesis).\footnote{See \cite{Neukirch1999}.} 
This gives the formula $\Diff(K/\QQ_p) = (e\pi^{e-1})$ from which we compute
$$\ord_p(\Diff(K/\QQ_p))= \ord_p(e) + (e-1)\ord_p(\pi)= \ord_p(e) + (e-1)\frac{1}{e}.$$

The \emph{discriminant} of an extension of fields $L/L_0$ is then defined to be the ideal-norm of the different: 
$$\Disc(L/L_0) = N_{L/L_0}(\Diff(L/L_0)) \vartriangleleft \OO_{L_0} $$ 	
We remark that $\prod_p p^{\ord_p \Diff(L/\QQ)} = \vert \Disc(L/\QQ) \vert^{1/[L:\QQ]}$.
This is helpful when thinking about (say) \eqref{E:prob-szpiro}.
In later sections we will make use of the notation $\diff(K/\QQ_p) = \ord_p \Diff(K/\QQ_p)$.

%----------------------------------------------------
\subsection{Explicit Chinese Remainder Formulas}\label{S:chinese}
%----------------------------------------------------
Fix a field $K_0$ and an algebraic closure $\Kbar_0$. 
Let $K_1,\ldots,K_m$ be finite extensions of $K_0$ sitting inside the common algebraic closure.
The isomorphism of rings 
\begin{equation}\label{E:isomorphism}
 K_1\otimes \cdots \otimes K_m \xrightarrow{\varphi} \bigoplus_{\psi \in \Phi} L_{\psi}
\end{equation}
will play an important role for us. 
We describe its ingredients:

\begin{itemize}
\item $\Phi \subset \bigoplus_{i=1}^m \Hom(K_i,\Kbar_0)$, is a complete system of representatives under the equivalence relation
$$(\psi_1,\ldots,\psi_m) \sim (\sigma\psi_1,\ldots,\sigma\psi_m)$$
for $\sigma \in G(\Kbar_0/K_0)$. 
\item For $\psi = (\psi_1,\ldots,\psi_m) \in \Phi$ we let $L_{\psi}$ be the compositum  
$$ L_{\psi} = \psi_1(K_1)\cdots\psi_m(K_m) \subset \overline{K_0}.$$
\item The isomorphism  $\varphi$ is defined via extending linearly the map
$$ \varphi(a_1\otimes \cdots \otimes a_m) = (\varphi_{\psi}(a_1\otimes\cdots\otimes a_m))_{\psi\in \Phi}, $$
where $\varphi_{\psi}(a_1\otimes \cdots \otimes a_m) = \psi_1(a_1)\cdots\psi_m(a_m)$.
\end{itemize}
We prove \eqref{E:isomorphism} is an isomorphism:
To see this we first note that  $\Spec(K_1\otimes \cdots \otimes K_m) = \prod_{i=1}^m \Spec(K_i)$ so the scheme is zero dimensional (and the spectrum of a product of fields).
	Each maximal ideal in the tensor product is the kernel of some map $\varphi:K_1\otimes\cdots\otimes K_m \to \Kbar_0$. 
	Two such maps have the same kernel if and only if they differ by an automorphism of $\Kbar_0$. 
	This explains the bijection between maximal ideals of the tensor product and $(\bigoplus \Hom(K_i,\overline{K_0}))/\sim$.
	Also, using the composition 
	$$ K_i \to K_1 \otimes \cdots \otimes K_m \to \Kbar_0 $$
	we see that any $\varphi$ induces $\psi_i:K_i \to \Kbar_0$ and we witness $\varphi$ as having the special form $\varphi(a_1\otimes\cdots\otimes a_m) = \psi_1(a_1) \cdots \psi_m(a_m)$.

%-------------------------------------------------------------
\subsection{Field Embeddings vs Choices of Roots}
%-------------------------------------------------------------
Let $K/K_0$ be a finite field extension. 
Write this as a primitive extension with $K = K_0(\alpha)$ and let $f(x)$ be the minimal polynomial of $\alpha$. 
Using this notation we can write down a bijection
 $$ \Hom_{K_0}(K,\Kbar_0) \xrightarrow{\sim} \lbrace \beta \in \Kbar_0: f(\beta)=0 \rbrace $$
 $$ \psi_0\mapsto \psi_0(\alpha).$$
Now, let $\Phi \subset \bigoplus_{i=1}^m \Hom_{K_0}(K,\Kbar_0)$ be a complete system of representatives for the equivalence relation $\sim$. 
Let $K_i = K_0(\alpha_i)$ where $\alpha_i$ has minimal polynomial $f_i(x)$. 
We can modify any $\psi = (\psi_1,\ldots,\psi_m)$ by some $\sigma \in G(\Kbar_0/\Kbar)$ with $\sigma \psi_1 = \id_{K_1}$ so that 
 $$ (\psi_1,\psi_2,\ldots,\psi_m) \sim (\id_{K_1},\psi_2',\ldots,\psi_m').$$
Such choices of $\psi$ will be called \emph{normalized} (for $K_1$) and 
a collection of embeddings $\Phi$ will be called \emph{normalized} if each element is normalized. 

Note that normalized $\Phi$ are in bijection with tuples of roots of the corresponding minimal polynomials. 
 $$\Phi \xrightarrow{\sim} \lbrace \vec{\alpha}'=(\alpha_2',\ldots,\alpha_m'): f_2(\alpha_2')=0,\ldots,f_m(\alpha_m')=0 \rbrace $$
 $$ (\id_{K_1},\psi_2,\ldots,\psi_m)=(\psi_1,,\psi_2,\ldots,\psi_m) \mapsto (\psi_2(\alpha_2),\ldots,\psi_m(\alpha_m)).$$
We record that the inverse map is given by 
 $$ \vec{\alpha}' \mapsto \psi_{\vec{\alpha}'} $$
where the components of $\psi_{\vec{\alpha'}}$ are the field embeddings uniquely determined by where they send the specified primitive element.
We will make use of this correspondence frequently.

\subsection{Notation for Quotients}
For a polynomial ring $R[x_1,\ldots,x_n]/I$ we will sometimes use the notation $\bar{x}_1,\ldots,\bar{x}_n$ to denote the images of $x_1,\ldots,x_n$ in the quotient.

%-----------------------------------------------------------
\subsection{Decomposition Comparisons}
%-----------------------------------------------------------
Given fields $K_i = K_0(\alpha_i)$ with minimal polynomials $f_i(x)$ for $i=1,\ldots,m$ and $\Phi$ a $K_1$-normalized system of embeddings we are interested in a description of the isomorphism \eqref{E:isomorphism} under the image of the base-change functor $\Kbar_0\otimes_{K_1} - $. 
This description will be used later in relating the tensor product of rings of integers to the ring of integers of tensor products.

First, we observe that $K_1\otimes \cdots\otimes K_m \cong K_1[x_2,\ldots,x_m]/(f_2,\ldots,f_m)$. 
This gives 
\begin{align*}
 \Kbar_0\otimes_{K_1} ( K_1\otimes \cdots \otimes K_m) &\cong \Kbar_0\otimes_{K_1} K_1[x_2,\ldots,x_m]/(f_2,\ldots,f_m) \\
 &\cong \Kbar_0[x_2,\ldots,x_m]/(f_2,\ldots,f_m) \\
 &\cong \bigoplus_{\vec{\alpha}'} \Kbar_0[x_2,\ldots,x_m]/(x_2-\alpha_2',\ldots,x_m-\alpha_m').
\end{align*}
Hence the isomorphism 
$$\Kbar_0 \otimes_{K_1}\varphi:\bigoplus_{\vec{\alpha}'} \Kbar_0[x_2,\ldots,x_m]/(x_2-\alpha_2',\ldots,x_m-\alpha_m') \to \Kbar_0 \otimes_{K_1} \bigoplus_{\psi \in \Phi} L_{\psi}. $$
is now seen to be given by 
 $$ (f(\bar{x}_2,\ldots,\bar{x}_m))_{\vec{\alpha}'} \mapsto (f(\alpha_2',\ldots,\alpha_m'))_{\psi_{\vec{\alpha}'}}.$$
The point here is that base changing to the algebraic closure splits fields and this allows us to work with roots of polynomials.

\iffalse 
The fact that this map is an isomorphism can be seen inductively. 
We have that $K_1\otimes \cdots \otimes K_m$ is a direct sum of fields. 
In the case $m=1$ the isomorphism is clear.
For the inductive step we can break down $K_1\otimes \cdots \otimes K_{m-1}\otimes K_m = ( \bigoplus_i F_i) \otimes K_m = \bigoplus_i( F_i \otimes K_m) = \bigoplus_i \bigoplus_j F_{j,i}$.
\fi

%-----------------------------------------------------------
\subsection{Idempotents and Differents}
%-----------------------------------------------------------
Let $K_1,\ldots,K_m$ be finite extensions of a field $K_0$ with $K_i = K_0(\alpha_i)$ and minimal polynomials $f_i$. 
If $K_1$ contains the Galois closures of $K_2,\ldots,K_m$ then the idempotents of $K_1\otimes \cdots \otimes K_m$ have the form 
\begin{equation}\label{E:idempotents}
 g_{j_2,\ldots,j_m} = \prod_{i=2}^m \frac{f_i(\bar{x}_i)}{(\bar{x}_i-\alpha_{i,j_i})}\frac{1}{f_i'(\alpha_{i,j_i})}.
\end{equation}
Here $K_1\otimes \cdots \otimes K_m = K_1[x_2,\ldots,x_m]/(f_2,\ldots,f_m)$ and 
 $$f_i(x) = (x-\alpha_{i,1})(x-\alpha_{i,2}) \cdots (x-\alpha_{i,n_i}). $$
Alternatively, we can fix some $\psi := (\psi_1,\ldots,\psi_m)$ a tuple of embeddings $\psi_i: K_i\to \Kbar_0$ and write
 $$ g_{\psi} =\prod_{i=2}^m \frac{f_i(\bar{x}_i)}{(\bar{x}_i-\psi_i(\alpha_i))}\frac{1}{f_i'(\psi(\alpha_i))}. $$

\begin{proof}
	We know that 
 $$K_1 [x_2,\ldots,x_m]/(f_2,\ldots,f_m) 
	= \bigoplus_{\vec{\alpha}' = (\alpha_2',\ldots,\alpha_m')} K_1[x_2,\ldots,x_m]/(x_2-\alpha_2',\ldots,x_m-\alpha_m'),$$
	so find the idempotents in this decomposition is the same as solving for $g_{\vec{\alpha}'}$ such that 
	$$ \begin{cases}
	 g_{\vec{\alpha}'} \equiv 1 \ \  (x_2 - \alpha_2',\ldots,x_m-\alpha_m'), &  \\
	 g_{\vec{\alpha}'} \equiv 0 \ \ (x_2-\beta_2',\ldots,x_m -\beta_m'), & \vec{\beta}'\neq \vec{\alpha}'
	 \end{cases}. 
	 $$
	 Since $f_i(x)/(x-\alpha_i') \to f_i'(\alpha_i)$ as $x\to\alpha_i$ by L'h\^{o}pital's rule (which by universality of the computation holds algebraically),
	 the element
	  $$\widetilde{g}_i(x) = \frac{f_i(x)}{(x-\alpha_i)}\frac{1}{f_i'(\alpha_i)}$$
	 has $\widetilde{g}_i(\alpha_i')=1$ and $\widetilde{g}_i(\beta_i') = 0$ for $\beta_i'\neq \alpha_i'$.
	 To obtain our result we just take the product of the $\widetilde{g}_i$ as in the statement of the result. 
\end{proof}

The relation between idempotents and differents now appears clear via formulas \eqref{E:diff-formula} in \S\ref{S:diff-and-disc} and \eqref{E:idempotents}.
%-----------------------------------------------------------
\subsection{ Rings of Integers of Tensor Products vs Tensor Products of Rings of Integers {\cite[Theorem 1.1]{IUT4}} }
%-----------------------------------------------------------
We now give the comparison of $T = \bigotimes_{i=1}^m \OO_{K_i}$ and $\OO_L$.
Here $\OO_L = \bigoplus_{\psi \in \Phi} \OO_{L_{\psi}}$ where $L = K_1\otimes \cdots\otimes K_m = \bigoplus_{\psi \in \Phi} L_{\psi}$.
We remind ourselves that $\OO_L$ is a $T$-algebra.
Here $\varphi:T \to \OO_L$ is given by (extending linearly)
 $$ \varphi(a_1\otimes \cdots \otimes a_m) = (\psi_1(a_1)\cdots \psi_m(a_m))_{\psi\in \Phi}.$$
For future reference we will let $\varphi_{\psi}$ denote the component of $\varphi$ in the $\psi$th factor. Explicitly, $\varphi_{\psi}(a_1\otimes \cdots \otimes a_m) = \psi_1(a_1)\cdots\psi_m(a_m)$ if $\psi=(\psi_1,\ldots,\psi_m)$.   

\begin{theorem}\label{T:descent}
	Let $K_1,\ldots,K_m$ be finite extensions of $\QQ_p$ sitting in a fixed algebraic closure.
	Let $T = \OO_{K_1}\otimes \cdots \otimes \OO_{K_m}$. 
	Let $L = \kappa(T) = K_1\otimes \cdots \otimes K_m$. 
	Let $k_i$ for $i=1,\ldots,m$ denote the respective residue fields of $K_i$. 
	If $\beta = 1 \otimes f_2'(\alpha_2) \otimes \cdots \otimes f_m'(\alpha_m)$ where $\OO_{K_i} = W(k_i)[\alpha_i]$ with Eisenstein polynomial $f_i(x) \in W(k_i)[x]$ then 
	 $$ \beta \in (T:_L \OO_L). $$
	That is $\beta \cdot \OO_L\subset T$.
\end{theorem}
\begin{proof}
	in what follows we let $\Zbar_p$ denote the integral closure of $\ZZ_p$ in $\QQbar_p$.
	In view of faithful flatness \cite[Chapter 3, exercises 16,17]{Atiyah1969} it is enough to show
	 $$ \Zbar_p \otimes_{\OO_{K_1}} (\beta \cdot \OO_L) \subset \Zbar_p \otimes_{\OO_{K_1}} T.$$
	We will use the notation $\OObar_L := \Zbar_p\otimes_{\OO_{K_1}} \OO_L$ and $\Tbar := \Zbar_p\otimes_{\OO_{K_1}} T$. 
	Using our embedding decomposition we have $\Zbar_p \otimes_{\OO_{K_1}}(\beta \OO_L) = \betabar\cdot \OObar_L$ where $\betabar = \sum_{\psi \in \Phi} \varphi_{\psi}(\beta) g_{\psi}.$ 
	Here we note that 
	$$\varphi_{\psi}(\beta)= \varphi_{\psi}(1 \otimes f_2'(\alpha_2) \otimes \cdots \otimes f_m'(\alpha_m)) = f_2'(\psi_2(\alpha_2))\cdots f_m'(\psi_m(\alpha_m)),$$
	for $\psi = (\psi_1,\ldots,\psi_m) \in \Phi$ (here we take $\Phi$ to be $K_1$-normalized).
	
	We now use that the idempotents are given by 
	 $$ g_{\psi} = \prod_{i=1}^m \frac{f_i(\bar{x}_i)}{(\bar{x}_i-\psi_i(\alpha_i))} \frac{1}{f_i'(\psi_i(\alpha_i))} \in \frac{1}{\varphi_{\psi}(\beta)}\Zbar_p[\bar{x}_2,\ldots,\bar{x}_m] = \frac{1}{\varphi_{\psi}(\beta)}\Tbar.$$
	
	Now we just check: if $x\in \overline{R}$ it has the form $x = \sum_{\psi\in \Phi} x_{\psi} g_{\psi}$ for some $x_{\psi} \in \Zbar_p$. 
	We have 
	\begin{align*}
	\betabar\cdot x &= \left( \sum_{\psi\in \Phi} \varphi_{\psi}(\beta)g_{\psi} \right) \left( \sum_{\xi \in \Phi} x_{\xi} g_{\xi} \right)\\
	&= \sum_{\psi} \varphi_{\psi}(\beta)x_{\psi}g_{\psi} \in \Tbar.
	\end{align*}
	The second equality follows from orthogonality of idempotents and the last membership statement follows from the fact that $\varphi_{\psi}(\beta)g_{\psi} \in \Tbar$.  
\end{proof}

\begin{remark}
	The proof of Theorem~\ref{T:descent} has nothing to do with $K_1$. 
	We can chose some $K_i$ which makes the inclusion tightest. 
\end{remark}

%%%%%%%%%%%%%%%%%%%%%%%%%%%%%%%%%%%%%%%%%%%%%%%%%%%%%%%%%%%%%%%%%%
\section{Conductors, Minimal Discriminants, and Ramification}\label{S:geometry-and-galois}
%%%%%%%%%%%%%%%%%%%%%%%%%%%%%%%%%%%%%%%%%%%%%%%%%%%%%%%%%%%%%%%%%%%

This section contains definitions and facts about bad reduction, minimal discriminants, and Galois theory necessary for our applications. 
Readers just interested in Probabilistic Szpiro or Baby Szpiro (\S\ref{S:prob-szpiro-sec}) may skip the last two subsections and proceed directly to \S\ref{S:hulls}.
For a quick reading, readers may which to skip to \S\ref{S:hulls} and come back to this section as needed in the course of reading \S\ref{S:prob-szpiro-sec} or \S\ref{S:explicit}. 

\subsection{Inertia/Decomposition Sequences}
Recall that for a finite extension $K$ of $\Q_p$ we have an extension of topological groups 
$$ 1 \to I_K \to G_K \to G_k \to 1,$$
where $I_K$ is the inertia group and $k$ is the residue field. 

If $L$ is a global field and $v \in V(L)$ is non-archimedean and $\vbar \vert v$ is a place of $\Lbar$ we have 
$$ 1 \to I(\vbar/v) \to D(\vbar/v) \to G(\kappa(\vbar)/\kappa(v)) \to 1 $$
where $G(\vbar/v) = \Stab_{G_L}(\vbar) \cong G_{L_v}$ is the decomposition group.

\subsection{Unramified and Ramified Representations} 
Let $L$ be a finite extension of $\QQ_p$. If $X$ is an object in a category, a representation $\rho: G_L \to \Aut(X)$ is \emph{unramified} if and only if $\rho(I_L)=1$.
We may speak of $X$ being unramified, where the representation is understood (usually torsion points of an elliptic curve).

Let $L$ be a global field.
Given $\rho: G_L \to \Aut(X)$ we say that $\rho$ is \emph{unramified} at $v$ if and only if $\rho\vert_{G_v}$ is unramified.
In either of these cases, if a representation is not unramified it is called \emph{ramified}.

\subsection{Good and Bad Reduction} Let $K$ be a finite extension of $\QQ_p$. 
Let $R = \OO_K$ be its ring of integers and let $k$ be its residue field.
Let $A_K$ be an abelian variety over $K$. 
We recall that $A_K$ has \emph{good reduction} if and only if there exists and abelian schemes $A$ over $R$ whose generic fiber is isomorphic to $A_K$.
This is equivalent to the special fiber of the N\'eron model being an abelian variety.
Given an abelian variety $A$ over a global field $L$ we say that $A_L$ has \emph{good reduction} at $v$ if and only if $A_{L_v}$ does.

\subsection{Division Fields and Galois Representations} 
Let $A$ be an abelian variety over a field $L$.
Let $m$ be an integer. 
We will abuse notation and let $A[m]$ denote both the group scheme of $m$-torsion points and the $G_L$-module given by taking $\Lbar$-points of this group scheme.

Assume now that $L$ is a number field. 
We will let $L_l = L(A[l])$. 
We remark that $L_l$ may be defined by literally adjoining the coordinates of torsion points in some model and that this field extension is independent of the model.
If we fix an algebraic closure $\Lbar$ we also have $L_l \cong \Lbar^{\ker\rho_l}$. 
We also note that $G(L_l/L) \cong \im(\rho_l: G_L \to \Aut(A[l]))$. 

Consider now the Tate module $T_lA=\varprojlim A[l^n]$ in the category of Galois modules. 
We let $\rho_{l^{\infty}}:G_L \to \Aut T_lA$ denote action in the underlying representation.
When it is necessary to specify the abelian variety we use $\rho_{l^{\infty},A}$.

Serre's surjectivity theorem says that for an elliptic curve without complex multiplication the image of $\rho_{l}$ surjective for all but finitely many $l$.
This implies that for $l$ sufficiently large $\im(\rho_l) \cong \GL_2(\FF_l)$.
\footnote{
	It is conjectured by Serre that for every number field $L$ there exist some $l_{\max}$ such that for every elliptic curve and all $l\geq l_{\max}$ that $\im(\rho_{E,l})=\GL_2(\FF_l)$.
	In the case $L=\QQ$ it is further conjectured that $\l_{\max}=37$. }
We make this remark because the initial theta data hypotheses of \cite[\S 5]{Dupuy2020a} require $\rho_l(G_F) \supset \SL_2(\FF_l)$ --- Serre's Conjecture says this is generically true.

%--------------------------------------------------
\subsection{Minimal Discriminants and Tate Parameters}
%--------------------------------------------------
We suppose that $E$ is an elliptic curve over a number field $F$ sitting in initial theta data. 
We will assume that the it is semi-stable (all bad places are places with multiplicative reduction).
Note that if it is not semi-stable one can make a finite change of base such that all places of the new field above a place of additive reduction in the old field are places of good reduction. 
Under any base change, places in the new field over places of multiplicative reduction in the old field still have multiplicative reduction (hence the word semi-stable).

In the case that $E$ is an elliptic curve over $L$, a finite extension of $\QQ_p$ by \cite[Ch V, Lemma 5.1]{Silverman2013} if $\vert j_E\vert_p>1$ (which is equivalent to $E$ having multiplicative reduction) there exists a Tate parameter $q=q_E \in \overline{L}$ and an isomorphism of elliptic curves $u: E \to E_q$ defined over $\overline{L}$.
Here $E_q$ is the Tate curve which admits a Tate uniformization.
Note that this implies that all elliptic curves without potential good reduction have a unique Tate parameter at bad places.
In fact: $q_E \in \Q_p(j_E)$ if $L$ is a finite extension of $\Q_p$.

The following describes the relationship between the minimal discriminant and the Tate parameter.
\begin{lemma}
	If $E$ is an elliptic curve over $L$ a complete discretely valued field with valuation $v$ then 
	\begin{enumerate}
		\item If $E$ has multiplicative reduction then $\ord_v(\Delta^{\min})=\ord_v(q_E)$, where $\Delta^{\min}$ is the minimal discriminant $E/L$. 
		\item All Tate curves $E_q$ are minimal Weierstrass models.
	\end{enumerate}
\end{lemma}
\begin{proof}
	The proof of the first assertion follows from Ogg's Formula.
	This formula states
	$$c = \ord(\Delta^{\min}) +1 -m.$$
	Here $c$ is the local conductor exponent, $\Delta^{\min}$ is the minimal discriminant and $m$ is the number of irreducible components $\mathcal{E}_s$ the special fiber of the N\'eron model of $E$ for $R=\OO_L$.
	Since our elliptic curve has multiplicative reduction this implies $c=1$ which implies $m=\ord_v(\Delta^{\min})$.
	Now we have
	$$ \mathcal{E}_s/\mathcal{E}_s^0 \cong E_q(L)/(E_q)_0(L) \cong (L^{\times}/q^{\ZZ})/(R^{\times}/q^{\ZZ}) \xrightarrow{\ord_v} \Z/\ord_v(q).$$
	The first equality is the Kodaira-N\'{e}ron Theorem, where $\mathcal{E}_s$ denotes the special fiber of the N\'{e}ron model and the superscript zero denotes the connected component of the identity. Also $(E_q)_0(L)$ is the kernel of specialization. 
	The second equality follows from Tate uniformization and the last equality follows from taking valuations.  
	From this the equality follows. 
	(see \cite[Appendix C]{Silverman2009}).
	
	We now prove $\Delta_{E_q}$ is minimal.
	We know that 
	 $$ \Delta_{E_q} = q \prod_{n\geq 1} (1-q^n)^{24}. $$
	This shows $\ord_v(\Delta_{E_q}) = \ord_v(q)$ and since $\ord_v(q) = \ord_v(\Delta^{\min}_{E_q})$ from the first assertion of the lemma we are done.
\end{proof}

%-----------------------------------------------
\subsection{Minimal Discriminants and Base Change}
%-----------------------------------------------
The following describes how minimal discriminants behave under base change. 
\begin{lemma}
Let $K/F$ be a finite extension of number fields. 
If $E/F$ is a semi-stable elliptic curve then 
$$[K:F] \ln \vert \Delta_{E/F}^{\min} \vert = \ln \vert \Delta^{\min}_{E_K/K} \vert. $$	
\end{lemma}
\begin{proof}
	The proof is a computation:
	\begin{align*}
	\ln \vert \Delta^{\min}_{E_K/K} \vert &= \sum_{w\in V(K)} \ord_w(\Delta^{\min}_{E_K/K}) f(w/p_w) \ln(p_w) \\
	&= \sum_{w\in V(K)} \ord_w(q_w) f(w/p_w) \ln(p_w) \\
	&= \sum_{v \in V(F)} \sum_{w \in V(F)_w} e(w/v) \ord_v(q_v) f(w/v)f(v/p_v) \ln(p_v) \\
	&= \sum_{v\in V(F)} \left( \sum_{w\vert v} [K_w:F_v] \right) \ord_v(q_v)f(v/p_v)\ln(p_v) \\
	&= [K:F] \sum_{v\in V(F)} \ord_v(q_v) f(v/p_v) \ln(p_v) = [K:F]\ln \vert \Delta^{\min}_{E/F} \vert.
	\end{align*}
\end{proof}

%----------------------------------------------
\subsection{Normalized Arakelov Degrees}
%----------------------------------------------
For a number field $L$ and an Arakelov divisors $D \in \Divhat(L)$ the \emph{normalized Arakelov degree} is defined by 
$$\deghatu_{L}(D) = \frac{\deghat_L(D)}{[L:\QQ]}.$$
For $v \in V(L)_0$ with $[v] \in \Divhat(L)$ degrees are normalized so that $\deghat([v]) = \ln \vert Nv \vert = f_v\ln(p_v)$ where $p_v$ is the characteristic of $\kappa(v)$ and $f_v$ is the inertia degree. 
We use the property that the normalized degree is invariant under pullback: if $f: V(L) \to V(L_0)$ is the natural map associated to an extension of number fields $L_0 \subset L$ and $D \in \Divhat(L_0)$ then 
$$ \deghatu_{L_0}(D) = \deghatu_L( f^*D). $$
We record that $f^*[v_0] = \sum_{v\vert v_0} e(v/v_0) [v]$.

%----------------------------------------------
\subsection{$q$ and Theta pilots}
%----------------------------------------------
Fix initial theta data $(\Fbar/F, E_F, l, \Mu, \Vu, V^{\bad}_{\mod}, \epsu)$. 
Furthermore, suppose it is built from the field of moduli so that $V^{\bad}_{\mod}\subset V(F_0)$ contains all the semi-stable places of bad reduction.
\begin{definition}
The \emph{$q$-pilot divisor} of this data is then 
\begin{equation}
P_q = \sum_{v \in V^{\bad}_{\mod}} \ord_v( q_v^{1/2l}) [v] \in \Div(F_0)_{\QQ}.
\end{equation}
\end{definition}

The $q$-pilot is related to the minimal discriminant of an elliptic curve by the following formula:
\begin{equation}
\deghatu_{F_0}(P_q) = \frac{1}{2l} \frac{\ln \vert \Delta^{\min}_{E/F} \vert }{[F:\QQ]}.
\end{equation}
To see this we perform a simple computation:
\begin{align*}
\deghatu_{F_0} \left( \sum_{v \bad} \ord_v(q_v) [v] \right) &= \deghatu_K\left( \sum_{v\in V^{\bad}_{\mod}} \ord_v(q_v) \sum_{w \in V(K)_v} e(w/v)[w] \right)\\
&= \deghatu_K\left( \sum_{v\in V^{\bad}_{\mod}} \sum_{w \in V(K)_v} \ord_w(q_v)[w] \right) \\
&= \deghatu_K \left( \sum_{w \bad} \ord_w(q_w) [w] \right)\\
&= \deghatu_K \left( \sum_{w \bad} \ord_w(\Delta^{\min}_{E_K/K})[w] \right) \\
&= \frac{\ln \vert \Delta^{\min}_{E_K/K} \vert }{[K:\QQ]} = \frac{\ln \vert \Delta^{\min}_{E/F} \vert }{[F:\QQ]}.
\end{align*}

We now discuss Theta pilots. 
\begin{definition}
The \emph{theta pilot divior} is a tuple $P_{\Theta} = (P_{\Theta,j})_{j=1}^{(l-1)/2} \in \Divhat_{\lgp}(F_0)_{\QQ}^{(l-1)/2}$ where 
$$ P_{\Theta,j} = \sum_{v \in V(F_0) \bad} \ord_v( q_v^{j^2/2l}) [v] \in \Divhat(F_0)_{\QQ}.$$	
\end{definition}
The relationship between the theta and $q$-pilots is given by 
\begin{equation}\label{E:trivial-relation}
 \deghat_{F_0}(P_q) = \frac{l(l+1)}{12} \deghat_{\lgp,F_0}(P_{\Theta}).
\end{equation}
This formula is derived by a simple computation:
\begin{align*}
\deghat_{\lgp,F_0}(P_{\Theta})&= \frac{2}{l-1} \sum_{j=1}^{(l-1)/2} \deghat_{F_0}(P_{\Theta,j}) \\
&= \frac{2}{l-1} \sum_{j=1}^{(l-1)/2} j^2 \deghatu_{F_0}(P_q) \\
&= \frac{l(l+1)}{12} \deghat_{F_0}(P_q).
\end{align*}

\begin{remark}\label{R:SS}
	\begin{enumerate}
	\item The assertion of \cite[pg 10]{SS} is that \eqref{E:trivial-relation} is the only relation between the $q$-pilot and $\Theta$-pilot degrees. The assertion of \cite[C14]{comments} is that \cite[pg 10]{SS} is not what occurs in \cite{IUT3}.  
	The reasoning of \cite[pg 10]{SS} is something like what follows:
	\begin{enumerate}
		\item The $\Theta_{\LGP}^{\times \mu}$-link in \cite{IUT3} is a polyisomorphism between $\mathscr{F}^{\Vdash\blacktriangleright \times \mu}$-strips, ${}^{0,0}\mathscr{F}^{\Vdash\blacktriangleright \times \mu}_{\LGP}$ and ${}^{1,0}\mathscr{F}^{\Vdash\blacktriangleright \times \mu}_{\Delta}$. 
		\item Within these objects there are two global realified Frobenioids ${}^{0,0}\mathcal{C}^{\Vdash}_{\LGP}$ and ${}^{1,0}\mathcal{C}^{\Vdash}_{\Delta}$.
		Also there exists objects ${}^{0,0}P^{\Vdash}_{\Theta} \in {}^{0,0}\mathcal{C}^{\Vdash}_{\LGP}$ and ${}^{1,0}P^{\Vdash}_{q} \in {}^{1,0}\mathcal{C}^{\Vdash}_{\Delta}$ called the (0,0) theta pilot object and (1,0) $q$ pilot object respectively and the theta link $\Theta^{\times \mu}_{\LGP}$ is such that $\Theta^{\times \mu}_{\LGP}({}^{0,0}P^{\Vdash}_{\Theta}) = {}^{1,0}P_{q}^{\Vdash}$.
		\item To each such global realified Frobenioids $\mathcal{C}^{\Vdash}$ we can interpret a one dimensional real vector space $\Pic(\mathcal{C}^{\Vdash})$. 
		Also, to any object $P^{\Vdash} \in \mathcal{C}^{\Vdash}$ there is an associated degree $\deg_{\mathcal{C}^{\Vdash}}(P^{\Vdash}) \in \Pic(\mathcal{C}^{\Vdash})$. 
		\item Any isomorphism between ${}^{0,0}\mathcal{C}^{\Vdash}_{\LGP}$ and ${}^{1,0}\mathcal{C}^{\Vdash}_{\Delta}$ induces an isomorphism between $\Pic({}^{0,0}\mathcal{C}^{\Vdash}_{\LGP})$ and $\Pic({}^{1,0}\mathcal{C}^{\Vdash}_{\Delta})$.
		\item An identification one can make is to fix isomorphisms
		 $$ \alpha: \Pic({}^{0,0}\mathcal{C}^{\Vdash}_{\LGP}) \to \RR$$
		 $$\beta: \Pic({}^{1,0}\mathcal{C}^{\Vdash}_{\Delta}) \to \RR $$
		specified by extending linearly
		 $$ \alpha(\deg_{{}^{0,0}\mathcal{C}^{\Vdash}_{\LGP}}({}^{0,0}P^{\Vdash}_{\Theta})) = \deghatu_{\lgp}(P_{\Theta}), $$
		 $$\beta(\deg_{{}^{1,0}\mathcal{C}^{\Vdash}_{\Delta}}({}^{1,0}P^{\Vdash}_{q})) = \deghatu(P_{q}),$$ 
		where the degree on the left hand side are as in the present subsection (\cite[2.1.6]{SS} calls this the canonical trivialization.)
		\item The authors of the present article, Scholze-Stix, and Mochizuki all agree that the items above lead to a contradition.
		Stripping away the abstraction, these assertions are tautologically equivalent to $\deghat_{\lgp,F_0}(P_{\Theta})=(l(l+1))/12) \cdot \deghat_{F_0}(P_q)$ and $\deghat_{\lgp,F_0}(P_{\Theta})=\deghat_{F_0}(P_q)$. This clearly gives a contradiction. 
		\item It is our understanding that no such $\alpha$ map is specified in IUT; meaning that commutativity of the diagram consisting of the map induced by $\Theta^{\times \mu}_{\LGP}$, $\alpha$, and $\beta$ is not asserted.
	\end{enumerate}
	\item We would like to point out that the diagram on page 10 of \cite{SS} is very similar to the diagram on $\S8.4$ part 7, page 76 of the unpublished manuscript \cite{Tan2018} which Scholze and Stix were reading while preparing \cite{SS}.
	\item As of August 1st 2019, the documents above can be found at \url{http://www.kurims.kyoto-u.ac.jp/~motizuki/IUTch-discussions-2018-03.html}. 
	We note that there is also the review \cite{Roberts2019} which some may find interesting.
	\end{enumerate}

\end{remark}
%-----------------------------------------------
\subsection{N\'eron-Ogg-Shafarevich: Conductors and Good Reduction}
%-----------------------------------------------
The following theorem of Serre and Tate, which they call the N\'eron-Ogg-Shafarevich Criterion, tells us how ramification of an $l$-power Tate module is related to the reduction geometry of the N\'eron model of corresponding the abelian variety.
\begin{theorem}[{\cite{Serre1968}}]
	Let $A$ be an abelian variety over a local field $L$ of residue characteristic $p$.
	The following are equivalent.
	\begin{enumerate}
		\item For all $m\in \NN$, $(m,p)=1$, $A[m]$ is unramified. 
		\item There exist a rational prime $l$ such that $l\neq p$ and $T_l(A)$ is unramified. 
		\item There exist infinitely many $m$ with $(m,p)=1$ such that $A[m]$ is unramified. 
		\item $A$ has good reduction.
	\end{enumerate}
\end{theorem}
We apply this in subsequent sections to get information about the behavior of ramification degrees in our computations.
We can apply this theorem to get a criteria relating the conductor of Abelian varieties to discriminants the of an associate $l$ division field. 

\begin{theorem}
	Let $A$ be an abelian variety over a number field $L$.
	Let $l$ be a rational prime. 
	Let $L_l = L(A[l])$.
	Let $w$ be a non-archimedean place of $L_l$ coprime to $l$ and $\ch(\kappa(w))=p$. 
	The following holds
	$$ e(w/p)>1 \iff w \vert l \mbox{ or } w \vert \Cond(A/L) \mbox{ or } w \vert \Diff(L/\QQ).$$
\end{theorem}
\begin{proof}
	Suppose that $e(w/p)>1$. 
	Since $e(w/p) = e(w/v)e(v/p)$, where $v\in V(L)$ is the image of $w\in V(L_l)$ under the natural map $V(L_l) \to V(L)$, we must have $e(w/v)>1$ or $e(v/p)>1$. 
	If $e(v/p)>1$ then $v \vert \Diff(L/\QQ)$ which implies $w\vert \Diff(L/\QQ)$.
	If $e(w/v)>1$ then $I_{w/v} \neq 1$ since $\#I_{w/v} = e(w/v)$. 
	Since $\rho_l: G(L_l/L) \to \Aut(A[l])$ is injective and $T_l$ has $A[l]$ as a quotient, we know that $\rho_l(I_{\vbar/v}) \neq 1$ and hence $T_lA$ is ramified. 
	This implies $v\vert \Cond(A/L)$.
	The final option is $w\vert l$. 
	
	Conversely suppose
	 $$ w \vert l \mbox{ or }  w \vert \Cond(A/L) \mbox{ or } w \vert \Diff(L/\QQ). $$
	
	If $w \vert l$ then since $L_l \supset \QQ(\zeta_l)$ we have $e(w/l) > l-1$. 
	If $w \vert \Diff(L/\QQ)$ then by definition $e(v/p)>1$. 
	If $w \vert \Cond(A/L)$ then $v \vert \Cond(A/L)$ since $L_l/L$ is Galois. 
	We know that 
	 $$ v \vert \Cond(A/L) \iff c_v \neq 0 \iff v \mbox{ is ramified } \iff I_{w/v} \neq 1. $$
	This proves the result. 
	Above, $c_v=\ord_v(\Cond(A/L)$.
\end{proof}

%%%%%%%%%%%%%%%%%%%%%%%%%%%%%%%%%%%%%%%%%%%%%%%%%%%%%%%%
\section{Estimates on $p$-adic Logarithms}\label{S:log-bounds}
%%%%%%%%%%%%%%%%%%%%%%%%%%%%%%%%%%%%%%%%%%%%%%%%%%%%%%%
The material in this section is applied in \S\ref{S:hulls} in order to obtain a bound on the smallest polydisc containing tensor products of log-shells.
In what follows we will let $\log$ denote the $p$-adic logarithm, $\ln$ denote the real valued natural logarithm, and $\log_p$ denote the real valued base $p$ logarithm.
We refer the reader to \cite{Robert2000} for a quick review of elementary properties of the $p$-adic logarithm.
See also \cite[\S 2]{Dupuy2020c}.

%------------------------------------
\subsection{Notation}
%------------------------------------

\subsubsection{} We let $\C_p$ be the $p$-adic completion of $\overline{\Q_p}$ and let $\ord_p$ be the unique extension of the valuation on $\Q_p$ to $\C_p$ with $\ord_p(p)=1$.
We normalize the $p$-adic absolute values by $\vert x \vert_p = p^{-\ord_p(x)}$. 

If $K$ is local field with uniformizer $\pi_K$ we let $\ord_K$ denote the valuation normalized by $\ord_K(\pi_K)=1$. 
In the case that $L$ is a global field and $v \in V(L)$ is a non-archimedean place, we let $\ord_v = \ord_{L_v}$ denote the normalized valuation on $L_v$.

\subsubsection{} Let $K/K_0$ be a finite extension of non-archimedean fields of residue characteristic $p$. 
We will let $e(K/K_0)$ denote the ramification degree of the extension. 
We will say $e(K/K_0)$ is \emph{small} provided $e(K/K_0)<p-1$. 
Note that small implies tame. 

If $L'\supset L$ is an extension of number fields and $v'\vert v$ are places of the respective number fields we let $e(v'/v) := e(L'_{v'}/L_v)$. 
If $L$ is a number field, we say that a non-archimedean place $v$ of $L$ is \emph{small} if $L_v/\QQ_p$ is small.

\subsubsection{} For a $p$-adic field $K$, $a\in K$ and $r\geq0$ a real number we will denote the closed disc of radius $r$ by
$$D_K(a,r) = \lbrace x \in K: \vert x \vert_p \leq r \rbrace.$$
Similarly if $L = \bigoplus_{j=1}^m L_j$ is a finite direct some of $p$-adic fields, $\vec{a} = (a_1,\ldots,a_m) \in L$ and $\vec{r} = (r_1,\ldots,r_m)$ is a vector of non-negative real numbers then we will denote the polydisc of polyradius $\vec{r}$ by 
$$D_L(\vec{a},\vec{r}) = \lbrace (x_1,\ldots,x_m) \in L: \vert x_1 \vert_p \leq r_1 \mbox{ and } \cdots \mbox{ and } \vert x_m \vert_p \leq r_m \rbrace. $$

When writing $D_L(0,R)$ where $R\in \RR$ we will understand this to mean $D_L(0,(R,R,\ldots,R))$.

%-----------------------------------------------------------
\subsection{Estimates on The Size of The $p$-Adic Logarithm}
%-----------------------------------------------------------
We begin by estimating the size of the $p$-adic logarithm (c.f. \cite[Prop 1.2]{IUT4}). 
\begin{lemma}[Crude Estimate]\label{L:log-bounds}
	Let $a\in \C_p$, with $\ord_p(a)>1$. 
	We have 
	 $$ \vert \log(1+a)\vert_p < \frac{c_p}{\ord_p(a)},$$
	where $c_p = (\exp(1)\ln(p))^{-1}$, where $\exp(1) = 2.71828182\ldots$ is the base of the natural log.
\end{lemma}
\begin{proof}
	To get an upper bound on $ \vert -\log(1-a)\vert_p = \vert  \sum_{n\geq 1} \frac{a^n}{n} \vert_p $
	for $\vert a \vert_p< 1$ it suffices to compute $\max \vert a^n/n \vert_p$. 
	Equivalently, we can compute the minimum of $\ord_p(a^n/n)$. 
	We find these lower bounds by using 
	$$\ord_p(a^n/n) = n\ord_p(a) - \ord_p(n)\geq n \ord_p(a) -\log_p(n),$$ 
	and minimizing the function 
	$$f(x) = xc - \log_p(x).$$
	The function has global minimum at $x_0 = 1/c\ln(p)$ which gives 
	$$ f(x) \geq f(x_0)=\frac{1}{\ln(p)}+\log_p(c\ln(p)). $$
	Converting this lower bound on the order to an upper bound on the $p$-adic absolute value gives our result. 
	\footnote{One could have also used $\ord_p(a^n/n) = p^m \ord_p(a) -m.$
		along the sequence $n = p^m$. This will give different, less useful bounds.
	See the remark below.} 
\end{proof}

\begin{remark}
	One can also minimize the function $f(x) = p^x c - x$ giving $\vert \log(1+a) \vert_p \leq b_p \frac{\vert a \vert_p }{\ord_p(a)}$, where $b_p=\frac{1}{\ln(p)e^{\ln(p)^2}}$. 
	This is not of any use to us.
\end{remark}

The application of Lemma~\ref{L:log-bounds} gives an upper bound on the smallest radius $r$ such that $\log(\OO_K^{\times}) \subset D_K(0,r)$ where $K$ is a finite extension of $\QQ_p$.
With knowledge that $e(K/\QQ_p)$ is small we can do much better. 
We state these results and omit the proofs.
\begin{lemma}\label{L:log-bounds2}
	Let $K/\Q_p$ be a finite extension. 
	\begin{enumerate}
		\item \label{I:easy-log} With no assumptions on the ramification of $K/\Q$ we have $\log(\OO_K^{\times}) \subset D_K(0, \frac{e(K/\Q_p)}{\ln(p)\exp(1)})$.
		\item If $e(K/\Q_p)<p-1$ then $\log(\OO_K^{\times}) = \pi \OO_K$ where $\pi$ is the uniformizer of $K$.
	\end{enumerate}
\end{lemma}

\begin{remark}
	In \cite[Prop 1.2]{IUT4} Mochizuki proves $\log(\OO_K^{\times}) \subset p^{-b}\OO_K$ where $b = \lfloor \ln(\frac{pe(K/\Q)}{p-1})/\ln(p) \rfloor - \frac{1}{e(K/\Q)}$. 
	As far as usability goes, the formula in Lemma~\ref{L:log-bounds2} while weaker, seems to be easier to understand. 
\end{remark}

\subsection{$p$-Adic Log Shells}\label{S:log-shells}
The present section collects and reformulates some of the material in \cite{AAG3}.
\begin{definition}
	Let $K/\Q_p$ be a finite extension. 
	The \emph{log-shell} of $K$ is the $\Z_p$-submodule of $K$ defined by $\Ical_K=\frac{1}{2p}\log(\OO_K^{\times})$
\end{definition}

\begin{lemma}[Upper Semi-Compatibility]
	$\Ical_K$ contains both $\OO_K$ and $\log(\OO_K^{\times})$. 
\end{lemma}
\begin{proof}
	It is clear that $\log(\OO_K^{\times}) \subset \Ical_K$. 
	Conversely, since $\vert 2p\vert_p < r_p$ we have $\log(1+2p\OO_K) = 2p\OO_K$ since $\ord_p(wp) > 1/(p-1)$.
	Hence 
	$$\Ical_K \supset \frac{1}{2p}\log(1+2p\OO_K) = \frac{1}{2p}(2p \OO_K) =\OO_K. $$
\end{proof}

\begin{remark}[Module Structures on $\log(\OO_K^{\times})$]
	For $K/\Q_p$ a finite extension we not that $\log(\OO_K^{\times})$ has the structure of an $\OO_K$-module very rarely. 
	In order for $\log(\OO_K^{\times})$ to be an $\OO_K$-module we need 
	$$ a \log(b) = \log(b^a) $$
	for $a\in \OO_K$ and $b\in \OO_K^{\times}$. 
	This in turn depends on the convervence of $b^a = \sum_{n=0}^{\infty} \frac{a(a-1) \cdots (a-n+1)}{n!}(b-1)^n$.
	We will not pursue this here, as estimates will not be needed. 
	On the other hand we do observe that $\log(\OO_K^{\times})$ is always a $\Z_p$-module for exactly the same reason.
\end{remark}

%%%%%%%%%%%%%%%%%%%%%%%%%%%%%%%%%%%%%%%%%%%%%%%%%%%%%%%%
\section{Archimedean Logarithms}\label{S:archimedean}
%%%%%%%%%%%%%%%%%%%%%%%%%%%%%%%%%%%%%%%%%%%%%%%%%%%%%%%
%Fix initial theta data $(\Fbar/F, E, l, \Mu, \Vu, V^{\bad}_{\mod}, \epsu)$. 
%Let $F_0 = \QQ(j_E)$ as usual.
%.
In order for our estimates to be complete we require definitions and estimates for $\hull(U_{\Theta})$ at the Archimedean factor $\LL_{\infty}$.
For $\vec{v} =(\vu_0,\ldots,\vu_j) \in V(F_0)_{\infty}^{j+1}$ we will let $H_{\vec{\vu}}$ denote the component of $\hull(U_{\Theta})$ in $K_{\vu_0}\otimes \cdots \otimes K_{\vu_j}$ (since $\sqrt{-1} \in K$ we know that $K_{\vu} \cong \CC$ for each $\vu \in \Vu$).

\begin{claim}\label{L:arch-log-bounds} 
	If $\vec{v} \in V(F_0)_{\infty}^{j+1}$ then $H_{\vec{\vu}} \subset D_{L_{\vec{\vu}}}(0; R_{\vec{\vu}})$ where $\ln( R_{\vec{\vu}} ) = (j+1) \ln(\pi)$.
\end{claim}
We do not develop the theory necessary to discuss this bound as this requires an Archimedean theory parallel to the $p$-adic theory in \cite{Dupuy2019b}. 
A full anabelian treatment requires so-called aut-holomorphic spaces.
The starting place is \cite[Definition 1.1]{IUT3}.
The claim above can be found in \cite[Proposition 1.5, Proof of Theorem 1.10, step vii]{IUT4}.

%%%%%%%%%%%%%%%%%%%%%%%%%%%%%%%%%%%%%%%%
\section{Upper Bounds on Hulls}\label{S:hulls}
%%%%%%%%%%%%%%%%%%%%%%%%%%%%%%%%%%%%%%%%
We now come to the section of the paper which contains the first major computation. 
Fix initial theta data $(\Fbar/F, E_F, l, \Mu, \Vu, V^{\bad}_{\mod}, \epsu)$.
In this section our goal is to find, for each prime $p$ and each $\vec{v} \in \coprod_{j=1}^{(l-1)/2} V(F_0)_p^{j+1}$, the smallest poly-disc $D_{L_{\vec{\vu}}}(0,R_{\vec{\vu}})$ such that the component of the multiradial representation at $\vec{\vu}$ is contained in this polydisc. 
The smallest possible polydisc here is called the hull. 
%--------------------------------
\subsection{Hulls}
%--------------------------------

If $L= \bigoplus_{j=1}^m L_j$ is a finite direct sum of $p$-adic fields and $\Omega \subset L$ then lets define $R_i(\Omega) = \max\lbrace \vert x_i \vert_p: (x_1,\ldots,x_m) \in \Omega \rbrace,$ 
then define the \emph{poly-radius of $\Omega$} to be 
$$ \vec{R}(\Omega) = (R_1(\Omega),\ldots,R_m(\Omega)). $$
Define the \emph{hull} of $\Omega$ to be the smallest poly-disc containing $\Omega$:
$$\hull(\Omega) = D_L(0, \vec{R}(\Omega)).$$

It is easy to check that if $\alpha  = (\alpha_1,\ldots,\alpha_m)\in L$ then 
$$\vec{R}(\alpha \cdot \Omega) = (\vert \alpha_1 \vert_p R_1(\Omega),\ldots, \vert \alpha_m \vert_p R_m(\Omega)).$$ 
Also, given a collection of compact regions $\Omega_i \subset L$ where $i=1,2,\cdots$  then for each $j$ where $1\leq j \leq m$ we have  
$$R_j( \bigcup_{i=1}^{\infty} \Omega_i ) = \sup\lbrace R_j(\Omega_i): i \geq 1 \rbrace.$$
Note that the right hand of the above equality is possibly infinite.

We now state some basic properties of hulls. 
For $A,B \subset L$ we will write 
$$A \subsetsim B \iff \hull(A) \subset \hull(B). $$
Note that $A\subsetsim B$ if there exists some $\QQ_p$-linear tranformation $T:L\to L$ with $\vert \det(T)\vert_p=1$ and $T(A) \subset B$ (such a $T$ could be multiplication by a unit of $L$ for example). 
Also if $\Omega \subset L$ and $a \in K_m$ (which we view as acting on $L$ via multiplication on the $m$th tensor factor) then $a^{\NN}\cdot \Omega \subsetsim a \cdot \Omega$.
To see that $\hull(a^{\NN} \cdot \Omega) \subset \hull(a\cdot \Omega)$, we observe that $a \in K_m$  acts on each direct summand of $L$ by $\psi_j(a)$ where we have written $L = \bigoplus L_{\psi_j}$ using the Chinese Remainder formulas developed in \S\ref{S:chinese}. 
This gives 
$R_j(a^{\NN}\cdot \Omega) = \sup\lbrace R_j(a^n\cdot \Omega) : n\geq 1 \rbrace = \vert a \vert_p R_j(\Omega)$. 
This implies $R_j(a^{\NN} \cdot \Omega) \leq R_j(a\cdot\Omega)$ and hence $\hull(a^{\NN}\cdot \Omega) \subset \hull(a\cdot\Omega)$.	

\subsection{Worst Case Scenario}\label{S:worst-case-scenario}
We now give a toy-version of our the computation of the hull bound associated to a tuple $\vec{v} \in V(F_0)^{j+1}$. 
Here we make assumptions on ramification of our fields. 

Let $K_1,\ldots,K_m$ be finite extensions of $\QQ_p$ (in our actual application $m$ will be $j+1$). 
Let $a \in K_m$ with $\vert a\vert_p <1$. 
Let $L = \bigotimes_{i=1}^m K_i \cong \bigoplus_{i=1}^r L_j$ where the factors of the right hand side come from the Chinese Remainder Theorem as in \S\ref{S:chinese}. 

In what follows we will let $\Ical = \bigotimes_{i=1}^m \Ical_{K_i}$ be the tensor product of log-shells and $\Aut(L:\Ical)$ denote the collection of $\QQ_p$-vector space automorphisms of $L$ obtained by extending $\QQ_p$-linearly $\ZZ_p$-lattice automorphisms of $\Ical$.
These automorphisms are a stand-ins for $\indone$ and $\indtwo$ in our actual applications (see \cite[\S4]{Dupuy2020a} for definitions).
\footnote{The only reason this subsection can't directly be applied is because the actual $\indone$ has some permutations among different tensor product factors of $\AA_{\Vu,p}^{\otimes j+1}$. 
The permutation of these factors does not appear in this example.}

This subsection gives a bound on the hull of the ``multiradial representation''\footnote{In \cite{Dupuy2020a} we used the notation $\Ubar$ for what we are now denoting $U$.}
$$U =\hull\left( \Aut_{\QQ_p}(L:\bigotimes_{i=1}^m \log(\OO_{K_i}^{\times})) \cdot (\OO_L^{\indthree(a)}) \right).$$
This region is a stand-in for the random measurable set $U^{(j)} \subset \AA_{\Vu,p}^{\otimes j+1}$ of the hull of the coarse multiradial representation of the Theta pilot region (see \cite[\S4]{Dupuy2020a} and the next section). 

We prove
\begin{equation}\label{E:in-disc}
	 \hull\left( \Aut(L:\Ical)  \cdot (\OO_L^{\indthree(a)}) \right) \subset D_L(0;R)
\end{equation} 
	where the radius $R$ is given by
\begin{equation}\label{E:radius}	
\ln(R) = - \lfloor \ord_p(a) + \Vert \diff \Vert_{\infty} - \Vert \diff \Vert_1 \rfloor \ln(p) + m\ln(c_p) + \sum_{i=1}^m \ln(e(K_i/\QQ_p)).
\end{equation}
The constant $c_p \in \RR$ and the vector $\diff \in \RR^m$ are given by 
	\begin{align*}
	c_p &= 1/\exp(1)\ln(p),\\
	\diff &= (\diff(K_1/\Q_p),\ldots,\diff(K_m/\QQ_p)).
	\end{align*}

To obtain this radius we compute. 
We have labeled each line in the computation below and give the justification for each step in the itemized environment following the displayed equations.
\begin{align}
&\Aut(L:\Ical)\left( a^{\NN} \cdot \left( \OO_L \cup \bigotimes_{i=1}^m \log(\OO_{K_i}^{\times}) \right) \right) \label{E:ind12} \\
&\subsetsim \Aut(L:\Ical)\left( a \cdot \left( \beta^{-1} \bigotimes_{i=1}^m \OO_{K_i} \cup \bigotimes_{i=1}^m \log(\OO_{K_i}^{\times}) \right) \right) \label{E:different} \\
&\subsetsim \Aut(L:\Ical)\left( a \beta^{-1} \Ical \right) \label{E:uppersemicompatibility} \\
&\subsetsim \Aut(L:\Ical) \left(p^{\lfloor \ord_p(a) - \ord_p(\beta) \rfloor} \Ical \right) \label{E:floors} \\
&= p^{\lfloor \ord_p(a) - \ord_p(\beta) \rfloor} \Ical \label{E:automorphisms} \\
&\subsetsim p^{\lfloor \ord_p(a) - \ord_p(\beta) \rfloor} D_L(0, \left( \frac{2p}{\exp(1)\ln(p)}\right)^m \prod_{i=1}^m e(K_i/\QQ_p) ) \label{E:log-bounds}
\end{align}

Since $\hull(D_L(0,R)) = D_L(0,R)$ for all radiuses $R>0$ we have 
 $$ \hull(U) \subset D_L(0, \frac{e(K_1/\Q_p) \cdots e(K_m/\Q_p)}{\vert 2 p \vert_p} p^{-\lfloor \ord_p(a) + \Vert \diff \Vert_{\infty} - \Vert \diff \Vert_1 \rfloor} ).$$

Here are the justifications for each step:
\begin{itemize}
	\item \eqref{E:in-disc} to \eqref{E:ind12}: Uses the main result concerning $\indthree$ in \cite{Dupuy2019b}
	\item \eqref{E:different}: Uses the theory of \S\ref{S:different}. 
	In particular there exist some $\beta = (\beta_1,\ldots,\beta_r) \in \bigoplus_{j=1}^r L_j=L$ such that $\beta \OO_L=\bigoplus_{j=1}^r \beta_j \OO_{L_j} \subset \bigotimes_{i=1}^r \OO_{K_i}$ where $\ord_p(\beta_j) = \Vert \diff \Vert_1 - \Vert \diff \Vert_{\infty}$ for each $j$ where $1\leq j\leq r$. 
	\item \eqref{E:uppersemicompatibility}: First we are using the ``upper semi-compatibility'' of $\Ical_K$, namely that for a finite extension $K$ of $\QQ_p$ we have $\log(\OO_K^{\times}), \OO_K \subset \Ical_K$.
	We use this fact tensor factor by tensor factor.
	Also, since the factors of $\beta$ all have large order, multiplication by $\beta^{-1}$ will increase the size of the hull.
	\item \eqref{E:floors}: We are using the general fact that if $A$ is a region and $\vert a_1 \vert_p < \vert a_2 \vert_p$ then $a_1 A \subsetsim a_2A$.
	\item \eqref{E:automorphisms}: This uses that $\Aut(L:\Ical)$ is by definition $\QQ_p$-linear and fixes $\Ical$ as a set. 
	\item \eqref{E:log-bounds}: We are applying the results of Lemma~\ref{L:log-bounds2}\ref{I:easy-log}.
\end{itemize}

\begin{remark}\label{R:ways-to-improve}
One can break this inclusion down in some alternative ways.
Here we highlight some areas for improvement.
We do not pursue these here. 
\begin{enumerate}
    \item Alternative to \eqref{E:ind12}: 
    For bounding $\Aut(L:\Ical)( a \cdot (\OO_L\cup \bigoplus_{i=1}^m \log(\OO_{K_i}) )$ one could write our an explicit $\ZZ_p$-basis for $a\cdot \OO_L$ and explicitly compute the action by $\Aut(L:\Ical)$. 
    \item Alternative to \eqref{E:different}: One could attempt to compute the index of $\bigotimes_{i=1}^m\OO_{K_i}$ in $\OO_L$. 
    This seems practical to do in specific toy cases but the size of the division fields may give in actual applications.
    It seems conceivable that other invariants around this inclusions can be used to write down more precise results. 
    \item \eqref{E:uppersemicompatibility}: One could attempt to find a smaller region here containing the two sets. 
    Are log-shells optimal? Maybe, maybe not.
    \item \label{I:ramification} \eqref{E:log-bounds}: We can go beyond the worst case scenario and make additional considerations about the ramification of the fields to improve bounds on $\Ical$. 
    This includes applying the second part of Lemma~\ref{L:log-bounds2} (which is applicable most of the times).
    In fact, for all but finitely many places of $v \in V(F_0)$ we have $\Ical_{\vu} = \OO_{K_{\vu}}$.
\end{enumerate}
\end{remark}

\subsection{Actual Scenario}
Fix initial theta data $(\Fbar/F, E_F, l, \Mu, \Vu, V^{\bad}_{\mod}, \epsu)$ built from the field of moduli.
In what follows $\AA_{\Vu}= \prod_{v \in V(F_0)} K_{\vu}$ denotes the ``fake adeles'' from \cite[\S3.1--\S3.3]{Dupuy2020a}. 
We seek to bound the sets $U^{(j)}_p \subset \AA_{\Vu,p}^{\otimes j+1}=:\LL_p^{(j)}$ where $U^{(j)}_p$ is of the form
$$U_p^{(j)} = \indtwo( \indone( \OO_{\LL_{p}^{(j)}}^{\indthree(\vec{a}_j)})).$$ 
Here we have made the following notational conventions:
\begin{align*}
\OO_{\LL_p^{(j)}} &= \bigoplus_{\vu \vert p} \OO_{\LL_{\vu}^{(j)}}, \\
\OO_{\LL_{\vu}^{(j)}} &= \Peel_{\vu}^j(\OO_{\LL_{\vu}^{(j)}}), \\
\OO_{\LL_p^{(j)}}^{\indthree(\vec{a}_j)} &= \bigoplus_{\vu \vert p} \OO_{\LL_{\vu}^{(j)}}^{\indthree(a_{j,v})},
\end{align*}
and we have let $\vec{a}_j = (a_{j,v})_{v\in V(F_0)}$ where 
$$ a_{j,v} = \begin{cases}
		q_{v}^{j^2/2l}, & v \mbox{ bad multiplicative} \\
		1, & \mbox{ else} \\
\end{cases}. $$
All of this of course depends on a choice of initial theta data. 
The peel decomposition $\Peel_{\vu}^j(\OO_{\LL_{\vu}^{(j)}})$ is described in \cite[\S3.3.7]{Dupuy2020a}.

Following Mochizuki we improve the toy bounds of \S\ref{S:worst-case-scenario} using Remark~\ref{R:ways-to-improve}\ref{I:ramification} (= considering what happens when ramification is small).
Let $\vec{\vu} = (\vu_0,\ldots,\vu_j) \in \Vu_p^{j+1}$.
We say that $\vec{\vu}$ is \emph{small} if every $e(\vu_i/p)$ is small for $0\leq i \leq j$.
Similarly we say that $\vec{\vu}$ is \emph{unramified} if $\vu_i$ is unramified for each $i$ where $0\leq i\leq j$.
We will also let $L_{\vec{\vu}} = K_{\vu_0} \otimes \cdots \otimes K_{\vu_j}$, where the tensor products are over $\QQ_p$.

\begin{lemma}\label{L:radius}
In the notation of this subsection, we have 
 $$\hull(U_p^{(j)}) \subset \prod_{\vec{\vu}\in \Vu_p^{j+1}} D_{L_{\vec{\vu}}}(0,R_{\vec{\vu}}) $$
where 
 $$
 \ln(R_{\vec{\vu}}) = 
 \begin{cases}
 0, &  \mbox{$\vec{\vu}$ unramified and $p\nmid \infty$} \\
 -\lfloor \ord_p(a_{j,v})-\ord_p(\beta_{\vec{\vu}}) \rfloor\ln(p), & \mbox{ $\vec{\vu}$ small and $p\nmid\infty$ }\\
 -\lfloor \ord_p(a_{j,v})-\ord_p(\beta_{\vec{\vu}}) \rfloor\ln(p) + (j+1)\ln(b_p) + \sum_{i=0}^j \ln(e(\vu_j/p)), & \mbox{ $p\mid \infty$ and $\vec{\vu}$ general } \\
 (j+1)\ln(\pi), & p\mid \infty
 \end{cases}  
 $$
\end{lemma}
\begin{proof}
	There are three points of departure from the computation in \S\ref{S:worst-case-scenario}: the specialization of $\beta$ and $a$, improvement of log-bounds, and the inclusion of the archimedean place. 
	In the case that $\vec{\vu}$ is unramified we know that $a_{j,v_j}=1$ by N\'{e}ron-Ogg-Shafarevich.
	In the case that $\vec{\vu}$ is small, we apply the bounds from Lemma~\ref{S:log-bounds}.
	In the archimedean case we apply Lemma~\ref{L:arch-log-bounds}.
\end{proof}

%%%%%%%%%%%%%%%%%%%%%%%%%%%%%%%%%%
\section{Probabilistic Versions of the Mochizuki and Szpiro Inequalities}\label{S:prob-szpiro-sec}
%%%%%%%%%%%%%%%%%%%%%%%%%%%%%%%%%%
Throughout this section we fix initial theta data $(\Fbar/F, E_F, l, \Mu, \Vu, V^{\bad}_{\mod}, \epsu)$ built from the field of moduli $F_0 = \QQ(j_E)$. 

\subsection{Probability Spaces}
Fix a rational prime $p$. 
Recall that, as in the introduction, we give $\coprod_{j=1}^{(l-1)/2} V(F_0)^{j+1}_p$ the structure of a finite probability space where 
$(v_0,v_1,\ldots,v_j) \in \coprod_{j=1}^{(l-1)/2} V(F_0)^{j+1}_p$ is assigned probability 
 $$ \Pr( (v_0,v_1,\ldots,v_j) ) = \frac{2}{l-1}\frac{[K_{\vu_0}:\QQ_p][K_{\vu_1}:\QQ_p]\cdots [K_{\vu_j}:\QQ_p] }{[F_0:\QQ]^{j+1}}.$$
The space $\coprod_{j=1}^{(l-1)/2} V(F_0)^{j+1}_p$ can be viewed as a uniform independent disjoint union of probability spaces $V(F_0)^{j+1}$. 
For a random variable $X(\vec{v})$ that depends on $\vec{v} = (v_0,v_1,\ldots,v_j) \in \coprod_{j=1}^{(l-1)/2} V(F_0)^{j+1}$ we can view the expectation of $X$ as an ``iterated expectation'', by first computing the expectation as we vary over $(v_0,\ldots,v_j) \in V(F_0)^{j+1}_p$ for a fixed $j$ and then computing the expection of these expectations as we vary uniformly over $j$.  
In what follows $\EE_p^2$ will denote this iterated expectation:
$$\EE_p^2(X(\vec{v})) = \EE( \EE( X(\vec{v}): \vec{v} \in V(F_0)^{j+1}): 1\leq j\leq (l-1)/2 ).$$ 
Note that the colons here do not denote conditional probabilities.

\subsection{Jensen's Inequality}
Jensen's inequality states that for a convex function $g(x)$ and a random variable $X$ that 
$$ g(\E(X)) \leq \E(g(X)). $$
The inequality goes the other way for concave functions and one can test for convexity using the second derivative test: a function of a real variable $g(x)$ is convex if and only $g''(x)\geq 0$. 
In particular $g(x) = \exp(x)$ is a convex function and $g(x) = \ln(x)$ in concave. 
This allows us to say that 
\begin{equation}\label{E:jensen-corollary}
\exp(\E(\ln(X))) \leq \E(X) \leq \ln( \E(\exp(X))).
\end{equation}

\subsection{Random Variables Pulled-back from a Projection}
Let $S$ be a discrete probability space.
Let $(X_1,\ldots,X_n)$ be a random variable on $S^n$. 
If $f(X_1,\ldots,X_n)$ only depends on $X_n$ (i.e. $f(X_1,\ldots,X_n) = g(X_n)$ for some function of a single variable $g$) then the expected value of $f(X_1,\ldots,X_n)$ can be computed by just varying over what the function depends on.
In symbols: 
$$\E( f(X_1,\ldots,X_n) ) = \E(g(X)).$$
It is also elementary to check that 
 $$\E(g(X_1)g(X_2)\cdots g(X_n)) = \E(g(X))^n.$$
 
\subsection{Measures} For $L$ a direct sum of $p$-adic fields, we will often make use of the formula 
$$\logmubar_{L}(D_L(0,R)) \leq \ln(R).$$
Here, for a finite dimensional vector space $V$ and a measurable set $A \subset V$ we define $\logmubar_V(A) = \ln(\mu_V(A))/\dim(V).$

\subsection{Probabilistic Mochizuki}
Using the Probabilistic formalism developed in \cite[\S 3.1.2]{Dupuy2020a}, we can state \cite[Corollary 3.12]{IUT3} in the following way:
\begin{theorem}[Tautological Probabilistic Inequality]\label{L:tautological}
	For $\vec{v} \in \Vu_p^{j+1}$ let 
	\begin{equation}\label{E:hull-radius}
	R_{\vec{\vu}}^{\circ}= \sup \lbrace R \in \RR: U_{\vec{\vu}} \subset D_{L_{\vec{\vu}}}(0,R) \rbrace,
	\end{equation}
	here $U_{\vec{\vu}}$ is the component of the multiradial representation in $L_{\vec{\vu}}$.
	Assuming \cite[Corollary 3.12]{IUT3} we have 
	\begin{equation}\label{E:tautological}
	  - \deghatu_{F_0}(P_q) \leq \sum_{p \in V(\QQ)} \EE_p^2( \ln R_{\vec{\vu}}^{\circ}). 
	\end{equation}
\end{theorem}
The radius $R_{\vec{\vu}}$ in Lemma~\ref{L:radius} gives an estimate on $R_{\vec{\vu}}^{\circ}$ in \eqref{E:tautological} giving
\begin{equation}\label{E:estimated}
- \deghatu_{F_0}(P_q) \leq \sum_{p \in V(\QQ)} \EE_p^2( \ln R_{\vec{\vu}}). 
\end{equation}
The rest of this subsection is devoted to estimating $\ln R_{\vec{\vu}}$ (so we will be deriving, in effect, will be estimates of estimates).
\begin{remark}
The computation of $R_{\vec{\vu}}$ is not optimal. It can be improved upon by readers in general or in special cases.
It is unclear how far off $R_{\vec{\vu}}^{\circ}$ is from $R_{\vec{\vu}}$. 
It would be interesting to develop a table of $R_{\vec{\vu}}^{\circ}$ in some numerical examples (if the computations involving the division fields are not prohibitively hard).
%Also, we remark that $R_{\vec{\vu}}^{\circ}$ in \label{E:hull-radius} can be replaced by a polyradius. 
%It is unclear how much this true polyradius 
\end{remark}

The readers should compare what follows to \cite[Proof of Theorem 1.10]{IUT4}. 
Fix $p \in V(\QQ)$.
We have  
\begin{equation}
\EE_p^2( \ln R_{\vec{\vu}}) \leq \Irm_p + \IIrm_p + \IIIrm_p + \IVrm_p + \Vrm_p
\end{equation}
where 
\begin{align*}
\Irm_p &=  -\EE_p^2\left (\ord_p(q_{v_j}^{j^2/2l}  ) \right)\ln(p) \\
\IIrm_p &=  \EE_p^2\left (  \Vert \diff_{\vec{\vu}} \Vert_1 - \Vert \diff_{\vec{\vu}} \Vert_{\infty} \right)\ln(p) \\
\IIIrm_p &= \EE_p^2\left( 1_{\ram}(\vec{\vu}) )\right) \\
\IVrm_p &= \EE_p^2( (j+1) \ln(b_p) 1_{\ram}(\vec{\vu}) )\\
\Vrm_p &= \EE_p^2\left(  \sum_{i=0}^j \ln(e(\vu_j/p)) \right)
\end{align*}
In the above formulas for $\IIIrm_p$ and $\IVrm_p$ the function $1_{\ram}(\vec{\vu})$ is the function which is $0$ if $\vec{\vu}$ is unramified and $1$ if $\vec{\vu}$ is ramified. \footnote{We say a tuple $(\vu_0,\ldots,\vu_j)$ is \emph{ramified} if there exists some $i$ with $0 \leq i \leq j$ such that $e(\vu_i/p)>1$.
If a tuple is not ramified it is called \emph{unramified}. }
We will denote the sums over $p$ of $\Irm_p$,$\IIrm_p$, $\IIIrm_p$, $\IVrm_p$, and $\Vrm_p$ by $\Irm$, $\IIrm,$ $\IIIrm$, $\IVrm$, and $\Vrm$ respectively. 

\begin{remark}
	At this stage we can already see Mochizuki's inequality as stated in \cite[\S2.12]{Fesenko2015}: we combine inequalities $-\deghatu(P_q) \leq \nlogn_{\LL}(\hull(U))$
	and $\nlogn_{\LL}(\hull(U)) \leq a(l)-b(l)\deghatu(P_q)$ to get  $(b(l)-1)\deghatu(P_q) \leq a(l)$ which gives 
	$$ \deghatu(P_q) \leq \frac{a(l)}{b(l)-1}.$$
    \cite[Claim 5]{SS} follows this style. 
	Further approximate computations can be found at \cite[slide 17]{Hoshi2017} (adapted in \cite[\S 1.3]{SS}). 
\end{remark}

%-------------------------------------------------------
\subsection{Computation of $\Irm_p$}
%-------------------------------------------------------
We have
\begin{align*}
\EE( \ord_p( \quu_{v_j}^{j^2} ): \vec{v} \in V(F_0)^{j+1} )  &= \EE(\ord_p(\quu_{v}^{j^2}): v \in V(F_0) )\\
&=\sum_{v\vert p} \frac{[F_{0,v}:\QQ_p]}{[F_0:\QQ_p]} \ord_p(\quu_{v}^{j^2}) \\
&= \frac{1}{[F_0:\QQ]} \sum_{v \in V(F_0)_p} e(v/p) \ord_p( \quu_{v}^{j^2}) f(v/p) \ln(p)\\
&= \deghatu( \sum_{v \in V(F_0)_p \ \bad } \ord_v( \quu_v^{j^2}) [v] ).
\end{align*}
Hence
\begin{equation}\label{E:formulaI}
\Irm_p =\sum_p \EE_p^2\left(\ord_p( \quu_{v_j}^{j^2})  \right) = \sum_p \EE( \deghatu( \sum_{v \in V(F_0)_p \ \bad } \ord_v( \quu_v^{j^2}) [v]) = \deghatu_{\lgp,F_0}(P_{\Theta}).
\end{equation}

%------------------------------------------------------
\subsection{Computation of $\IIrm_p$}
%------------------------------------------------------
In what follows we make use of the \emph{average different order of $\Vu$ over $p$} is defined to be the quantity
\begin{equation}
\overline{\diff}_{p}:= \log_p(\E(p^{\diff(\vu/p)})).
\end{equation}
We will prove 
\begin{equation}\label{E:formulaII}
	\IIrm_p \leq \frac{(l+1)}{4} \overline{\diff}_{p}\ln(p).
\end{equation}
Note that if we define the \emph{average different} for $\Vu$ by $ \overline{\Diff}(\Vu/\QQ) = \prod_p p^{\overline{\diff}_p} $
we get 
\begin{equation}
\IIrm \leq \ln \Diffbar(\Vu/\QQ).
\end{equation}

Before establishing \eqref{E:formulaII} it is convenient to make the following Lemma.
\begin{lemma}\label{L:part2-lemmas}
	For $\vec{v}\in V(F_0)_p^{n}$ let 
	$$ \diff_{\vec{\vu}} = \diff_{(\vu_1,\ldots,\vu_n)} = (\diff(\vu_1/p),\ldots,\diff(\vu_n/p)). $$ 
	For $\vec{v} \in V(F_0)_p^n$ following inequalities hold.
	\begin{enumerate}
		\item \label{I:averages} $\Vert\diff_{\vec{\vu}}\Vert_1 - \Vert \diff_{\vec{\vu}} \Vert_{\infty} \leq \frac{n-1}{n} \Vert \diff_{\vec{\vu}} \Vert_1.$
		\item \label{I:application-of-jensen} $\E(\diff_{(\vu_1,\ldots,\vu_n)}) \leq n \overline{\diff}_{p}.$ 
	\end{enumerate}
	The subscripts $1$ and $\infty$ denote the usual $l^1$ and $l^{\infty}$ norms for vectors in $\RR^n$.

\end{lemma}
\begin{proof}
\begin{enumerate}
	\item 	The proof is a fortiori. 
	For positive real numbers $a_1,\ldots, a_n$
	we have 
	\begin{align*}
	n(\sum_{i=1}^n a_i - \max_{1\leq i \leq n} a_i) =&n(\sum_{i=1}^n a_i) - n\max_{1\leq i \leq n} a_i)\\
	\leq& n(\sum_{i=1}^n a_i) - \sum_{i=1}^n a_i \\
	=& (n-1) \sum_{i=1}^n a_i.
	\end{align*} 
	This proves $\Vert \vec{a} \Vert_{1} - \Vert \vec{a} \Vert_{\infty} \leq \frac{n-1}{n} \Vert \vec{a} \Vert_1,$
	if we let $\vec{a} = (a_1,\ldots,a_n)$. 
	\item We will apply Jensen's inequality, to turn an expectation of a sum $\E(\diff(\vu_1/p) + \cdots + \diff(\vu_n/p))$
	into (the log of) an expectation of a product $ \E(p^{\diff(\vu_1/p)} \cdots p^{\diff(\vu_n/p)}).$
	Now that this is a product of random variables the expectation factors, 
	namely, $\E(p^{\diff(\vu_1/p)} \cdots p^{\diff(\vu_n/p)}) = \E(p^{\diff(\vu/p)})^n. $
	This shows $ \E(\diff_{(\vu_1,\ldots,\vu_n)}) \leq n\log_p \E(p^{\diff(\vu/p)})$ which is our desired result. 
\end{enumerate}
\end{proof}

We now prove our desired formulas:
	\begin{align*}
	\E(\Vert \diff_{(\vu_0,\ldots,\vu_j)} \Vert_1 - \Vert \diff_{(\vu_0,\ldots,\vu_j)}\Vert_{\infty}) &\leq \frac{j}{j+1}\E( \Vert \diff_{(\vu_0,\ldots,\vu_j)} \Vert_1)\\
	&\leq \frac{j}{j+1}( (j+1) \overline{\diff}_p ) \\
	&= j \overline{\diff}_p.
	\end{align*}
	The first line follows from Lemma~\ref{L:part2-lemmas}\ref{I:averages} and the second line follows from Lemma~\ref{L:part2-lemmas}\ref{I:application-of-jensen} (which as an application of Jensen's inequality together with the way expectations of products of random variables behave). 
	It remains to compute the expectation of these over $\lbrace 1,\ldots,j \rbrace$. 
	We have
	$$\E^2( \Vert \diff_{(\vu_0,\ldots,\vu_j)}\Vert_1 - \Vert \diff_{(\vu_0,\ldots,\vu_j)}\Vert_{\infty} ) \leq \E( j \overline{\diff_p})\\
	=\left(\frac{2}{l-1} \sum_{j=1}^{(l-1)/2} j \right)\overline{\diff}_p = \frac{l+1}{4}\overline{\diff}_p,$$
	which gives our result. 

%----------------------------------------------
\subsection{Computation of $\IIIrm_p$}\label{S:probability-of-ramification}
%----------------------------------------------
In what follows we will make use of the probability of a place $\vu \in \Vu_p$ to be unramified.
In formula this probability is defined by 
\begin{equation}
\PP_{\unr,p} = 1-\EE( 1_{\ram}(\vu): v \in V(F_0)_p).
\end{equation}
Also recall that since a tuple $\vec{\vu} = (\vu_0,\ldots,\vu_j) \in \Vu_p^{j+1}$ is unramified if and only if each $\vu_i$ is unramified for $0 \leq i \leq j$ this means that $\EE(1_{\ram}(\vu_0,\ldots,\vu_j)) = 1 - \PP_{\unr,p}^{j+1}.$
This then gives 
 $$ \IIIrm_p = \EE_p^2(1_{\ram}(\vu_0,\ldots,\vu_j)) = \frac{2}{l-1}\sum_{j=1}^{(l-1)/2}\left( 1-\PP_{\unr,p}^{j+1} \right)=1-\frac{2}{l-1}\sum_{j=1}^{(l-1)/2} \PP_{\unr,p}^{j+1}.$$
As the smallest of the $\PP_{\unr,p}^{j+1}$ is $\PP_{\unr,p}^{(l+1)/2}$ we get the following inequality:
\begin{equation}\label{E:formulaIII}
 \IIIrm_p \leq 1-\PP_{\unr,p}^{(l+1)/2}.
\end{equation}

%----------------------------------------------
\subsection{Computation of $\IVrm_p$}
%----------------------------------------------
We will prove 
 \begin{equation}\label{E:formulaIV}
 \IVrm_p \leq \frac{l+5}{4}\ln(b_p)\left( 1- \PP_{\unr,p}^{l+1/2} \right).
 \end{equation}
%\taylor{check the average of $j+1$ again.}
Using identical reasoning to \S\ref{S:probability-of-ramification} the first average is $\EE((j+1)\ln(b_p)1_{\ram}(\vec{\vu})) = (j+1) \ln(b_p)(1-\PP_{\unr,p}^{j+1})$. 
This gives
\begin{align*}
\IVrm_p &= \ln(b_p)\left( \frac{2}{l-1} \sum_{j=1}^{(l-1)/2} (j+1)(1 - \PP_{\unr,p}^{j+1}) \right) \\
&\leq \ln(b_p)( 1-\PP_{\unr,p}^{l+1/2})\left(\frac{l+5}{4} \right).
\end{align*}

%------------------------------------------
\subsection{Computation of $\Vrm_p$}
%------------------------------------------
It will be convenient to define $\overline{e}_p$, the \emph{average ramification index} of $\Vu$ over $p$. 
In notation it is defined by
\begin{equation}
 \overline{e}_p := \EE( e(\vu/p): v \in V(F_0)_p ).
\end{equation}
We now compute $\Vrm_p$: we have
\begin{align*}
\EE(\sum_{i=0}^j \ln(e(\vu_i/p))) &= \EE( \ln( \prod_{i=0}^k e(\vu_i/p))) \\
&\leq \ln( \EE\left(  \prod_{i=0}^j e(\vu_i/p)\right))\\
&\leq \ln( \EE(e(\vu/p))^{j+1}) = (j+1)\ln(\overline{e}_p)
\end{align*}
The first to second line is an application of Jensen's inequality and the second to third line uses that, for independent random variables, the expectation of the product is the product of the expectations.
We then can compute the second expectation by computing the uniform average of $j+1$ over $\lbrace 1, \ldots, (l-1)/2\rbrace$.
This gives
\begin{equation}\label{E:formulaV}
\IVrm_p \leq \frac{l+5}{4}\ln(\overline{e}_p).
\end{equation}
%\taylor{Check the average of $j+1$ again. This is the coeff.}

%-------------------------------------------------
\subsection{Archimedean Contribution}\label{S:arch-contr}
%-------------------------------------------------
Here we only have to deal with log-shells. 
Applying Lemma~\ref{L:arch-log-bounds} we have 
$$\E_{\infty}^2 = \frac{2}{l-1} \sum_{j=1}^{(l-1)/2} (j+1) \ln(\pi)  = \frac{l+5}{4}\ln(\pi).$$

%-------------------------------------------------
\subsection{Probabilistic Szpiro}\label{S:second-probabilistic}
%-------------------------------------------------
We now combine the results of the previous subsections. 
The verification of the following identities requires some careful bookkeeping. 

\begin{theorem}[Probabilistic Szpiro]\label{T:probabilistic2}
Assume \cite[Corollary 3.12]{IUT3} and Claim \ref{L:arch-log-bounds}.
Then for any elliptic curve $E/F$ in initial theta data $(\Fbar/F, E_F, l, \Mu, \Vu, V^{\bad}_{\mod}, \epsu)$ built over the field of moduli we have  
	\begin{equation}	
	\frac{1}{6+\varepsilon_l}\frac{\ln \vert \Delta^{\min}_{E/F} \vert }{[F:\QQ]} \leq \ln \overline{\Diff}(\Vu) + \sum_{p} \ln(\overline{e}_p)+A_{l,\Vu}
	\end{equation}
	where 
	 $$A_{l,\Vu}= \ln(\pi) + \sum_p (1 - \PP_{\unr,p}^{(l+1)/2})\left (\ln(b_p)+\frac{5}{l+4} \right ),$$
	and $b_p = 1/\exp(1)\ln(p)$, and $\varepsilon_l = 24(l+3)/(l^2+l-12)$.
\end{theorem}
\begin{proof}
	For the most part, this is just a combination of the bounds on $\Irm$,$\IIrm$,$\IIIrm$,$\IVrm$ and $\Vrm$ given by equations \eqref{E:formulaI},\eqref{E:formulaII}, \eqref{E:formulaIII}, \eqref{E:formulaIV}, and \eqref{E:formulaV}.
	The most interesting aspect of this computation is the appearance of the $6+\varepsilon_l$. 
	
	From the Tautological Probabilistic Inequality we get 
	\begin{align*}
	-\deghatu_{F_0}(P_q) \leq& - \deghatu_{\lgp,F_0}(P_{\Theta}) + \frac{l+1}{4} \ln \overline{\Diff} \\
	&+ \sum_p (1 - \PP_{\unr,p}^{l+1/2})(1+\frac{l+5}{4}\ln(b_p)) + \frac{l+5}{4}\sum_p \ln(\overline{e}_p) \\
	&+ \frac{l+5}{4}\ln(\pi).
	\end{align*}
	Using that $\deghatu_{\lgp,F_0}(P_{\Theta}) = ((l+1)l/12) \deghatu(P_q)$ and $\deghatu(P_q) = \ln \vert \Delta^{\min}_{E/F}\vert/2l[F:\QQ]$ we get
	\begin{align*}
	\left( \frac{(l+1)l}{12} -1 \right)\frac{1}{2l} \frac{\ln \vert \Delta^{\min}_{E/F} \vert }{[F:\QQ]} 
	\leq& \frac{l+1}{4}\ln \overline{\Diff} + \sum_p (1 - \PP_{\unr,p}^{l+1/2})(1+\frac{l+5}{4}\ln(b_p)) \\
	& + \frac{l+5}{4}\sum_p \ln(\overline{e}_p) + \frac{l+5}{4}\ln(\pi)
	\end{align*}
	We now divide both sides by $(l+5)/4$ an massage the algebra to get our result. 
	A simple computation shows that 
	 $$ \left( \frac{l(l+1)}{12} -1 \right) \left ( \frac{1}{2l} \right) \left( \frac{4}{l+5} \right ) = \frac{1}{6+\varepsilon_l} $$
	where 
	$$ \varepsilon_l = \frac{24l+72}{l^2+l-12}.$$
	This proves the assertion that $\varepsilon_l = O(1/l)$ as $l\to \infty$. 
	Finally, putting everything together we get 
	\begin{equation}
	\frac{1}{6+\varepsilon_l}\frac{\ln \vert \Delta^{\min}_{E/F} \vert }{[F:\QQ]} \leq \ln \overline{\Diff} + \sum_{p} \ln(\overline{e}_p) + A_{l,\Vu}.
	\end{equation}
	Here $ A_{l,\Vu}$ is as described in the statement of the proposition.
\end{proof}

\iffalse 
\begin{remark}
	An interesting part about the inequality is the factor $A_l$. 
	Its nonconstant part is made up of two terms which are roughly proportional 
	$\# \lbrace p : \exists \vu\in\Vu_p \mbox{ such that } e(\vu/p)>1 \rbrace $. 
	The interesting part is that the main contribution of this is negative as $b_p<1$.
	That means this term sort-of takes away from the logarithmic ramification indices. 
\end{remark}
\fi
%-------------------------------------
\subsection{Baby Szpiro}
%-------------------------------------

To demonstrate the utility Probabilistic Szpiro we give a ``Baby'' Szpiro inequality. 

\begin{theorem}[Baby Szpiro]\label{T:baby-szpiro}
	Assume \cite[Corollary 3.12]{IUT3} and Claim \ref{L:arch-log-bounds}.
	For an elliptic curve $E$ over a field $F$ sitting in initial theta data $(\Fbar/F, E_F, l, \Mu, \Vu, V^{\bad}_{\mod}, \epsu)$ built from the field of moduli we have  
	\begin{equation}
	\frac{1}{6+\varepsilon_l} \frac{\ln \vert \Delta^{\min}_{E/F}\vert }{[F:\Q]} \leq \ln( [K:\Q]^{5/4}) \ln( \vert \Disc(K/\Q)\vert^{5/4}) + \ln(\pi),
	\end{equation}
	here $\varepsilon_l = (24l+72)/(l^2+l-12)$.
\end{theorem}
\iffalse
The sum of j+1 is (l-1)(l-5)/8
multiplying by 2/(l+1) given 
We then use long division and the fact that l>3.

 $\frac{2\E_{\infty}^2}{l+1} \leq l\ln(\pi)$
\fi

\iffalse 
We first give a Lemma. 
\begin{lemma}
	$$\ln(\overline{\Diff})  \leq \ln( \rad \vert \Disc(K/\Q) \vert \cdot [K:\Q])$$
\end{lemma}
\fi

\begin{proof}[Proof of Baby Szpiro]
	This is just a simple application of the Probabilistic Szpiro for theta data (Theorem~\ref{T:probabilistic2}) using elementary bounds for the right hand side. 
	We use 
	\begin{align}
	&\ln(\overline{\Diff})  \leq \ln( \rad \vert \Disc(K/\Q) \vert \cdot [K:\Q]), \label{E:Diff-bound}\\
	&\sum_p \ln \overline{e}_p \leq \ln([K:\Q]) \omega( \vert \Disc(K/\Q) \vert),\\
    &\sum_p\left( 1- \PP_{\unr,p}^{\frac{l+1}{4}}\right)\left(\ln(b_p) + \frac{4}{l+1} \right) \leq \ln \vert \Disc(K/\Q) \vert, 
	\end{align}
	where $\rad(N) = \prod_{p\vert N} p$ and $\omega(N) = \sum_{d\vert n} 1$ is the ``number of divisors'' function.
	\footnote{ 
			The inequality \eqref{E:Diff-bound} is an application of the bounds on the different order given in \S\ref{S:diff-and-disc}.
			We have
			\begin{align*}
			p^{\diffbar_p} = \E(p^{\diff(\vu/p)}) &= \sum_{v\vert p } \frac{[F_{0,v}:\QQ_p]}{[F_0:\QQ]} p^{\diff(\vu/p)} \leq  \sum_{v\vert p } \frac{[F_{0,v}:\QQ_p]}{[F_0:\QQ]} p^{1-\frac{1}{e(\vu/p)} + \ord_p e(\vu/p)} \leq p\cdot  p^{\ord_p[K:\QQ]}.
			\end{align*}
			We then have 
			$$ \Diffbar(\Vu/\QQ) \leq \prod_p p \cdot p^{\ord_p[K:\QQ]} = \rad( \vert \Disc(K/\QQ) \vert) [K:\QQ], $$
			which gives the result.
	}
	These together with the bounds $\omega(N) \leq \ln(N)/\ln^2(N)$ give that the right hand side of the second probabilistic Szpiro is less than 
	 $$ \ln(Dd) + \ln(D)\ln(d) + \ln(D)$$
	where $D = \vert \Disc(K/\Q) \vert$ and $d = [K:\Q]$.
	This simplifies to %\footnote{	
%$\leq 2 \ln(D) + \ln(D)\ln(d) + \ln(d) =\ln(D)(2 + \ln(d)) + \ln(d) \leq \ln(D)(2+\ln(d)) + (2+\ln(d)) \leq(\ln(d)+1)(\ln(D)+2)$
%	}
$$\ln(Dd) + \ln(D)\ln(d) + \ln(D) \leq (\ln(D)+2)(\ln(d)+2).$$

    Since $D,d\geq \SL_2(\FF_l)\geq 6840$, we will find an $a\in \Q$ such that $\ln(x^a) \geq \ln(x) + 2$. 
    Solving the inequality gives $a \geq 2/\ln(x) + 1$ and since
     $$ \frac{2}{\ln(x)}+1 \leq \frac{2}{\ln(6840)}+1 \leq \frac{5}{4} $$
    we get that $(\ln(D)+2)(\ln(d)+2) \leq \ln(D^{5/4})\ln(d^{5/4})$ which proves the result.
\end{proof}

\iffalse 

Are there counter-examples to the inequality for 

\begin{remark}
	Here is a variant: if one has a priori lower bounds on $D$ and $d$ on can use these by writing 
	$$(\ln(D)+2)(\ln(d)+2) = (1+\frac{2}{\ln(D)})(1+\frac{2}{\ln(d)})\ln(D)\ln(d)$$
	and substituting them to get our desired inequality. 
	These numbers are precisely the $5/4$ bounds that we derived before. 
\end{remark}

$(5/4)^2 \approx 1+1/2$, does this guarantee ``weird ramification''. 
\fi

%%%%%%%%%%%%%%%%%%%%%%%%%%%%%%%%%%
\section{Deriving Explicit Constants For Szpiro's Inequality From Mochizuki's Inequality}\label{S:explicit}
%%%%%%%%%%%%%%%%%%%%%%%%%%%%%%%%%%

In order to get strong uniform versions of Szpiro's inequality from Mochizuki's inequality one needs to do some careful ramification analysis based on the N\'{e}ron-Ogg-Shafarevich Criterion (\S\ref{S:geometry-and-galois}). 
This process is a sort-of ``separation of variables'' writing an upper bound for $\vert \Disc(K/\QQ)\vert$ in terms of $d_0,l,\vert \Disc(F/\QQ)\vert$.

Here are the questions we needs to answer: What makes a place $w \in V(K)_p$ ramify?
What is the maximum possible ramification index $e(w/p)$ as we vary over $w \in V(K)_p$?
Does the size of $p$ matter?
We answer all of these questions in the subsequent section and apply these results to get our version of uniform Szpiro with exponent 24.  
%------------------------------
\subsection{Ramification Analysis}\label{S:ramification-analysis}
%------------------------------
The following Lemma answers the question about the maximal ramification index. 
\begin{lemma}
For every place $w \in V(K)$ we have 
 $$ e(w/p) \leq B_{l,d_0}$$
where $B_{l,d_0} =276480l^4d_0$.
\end{lemma}
\begin{proof}
Fix $w \in V(K)$. 
We consider the successive extensions\footnote{These come from the hypotheses of ``initial theta data built from the field of moduli''.}
 $$ K \supset F = F_2'(E[15]) \supset F_2'= F_0(E[2],\sqrt{-1}) \supset F_0 \supset \QQ. $$
We label the various images of $w$ under the induced map on places as follows:
 $$V(K) \to V(F) \to V(F_2') \to V(F_0) \to V(\QQ) $$
 $$ w \mapsto v \mapsto v_2' \mapsto v_0 \mapsto p. $$
In this notation we have 
\begin{align*}
  e(w/p) &= e(w/v)e(v/v_2')e(v_2'/v_0)e(v_0/p) \\
  &\leq l^4 \cdot 23040 \cdot 12 \cdot [F_0:\QQ]\\
  &= 276480l^4d_0 =: B_{l,d_0}.
\end{align*}
We explain these inequalities: rach of the extensions (other than $F_0 \supset \QQ$) is Galois and we have $G(K/F) \subset \GL_2(\FF_l)$,  $G(F/F_2) \subset \GL_2(\ZZ/15)$, and $\# G(F_2'/F_0) \vert 12.$
Knowing that $\#\GL_2(\FF_q)=q(q-1)(q^2-1)$ and plugging in explicit values gives the result.
\end{proof}

As a Corollary we get the following.
\begin{lemma}\label{L:most-ramification-small}
	If $p>B_{l,d_0}$ then $e(w/p)<p-1$. Note that this implies the ramification of $K/\QQ$ is small for all but finitely many places. 
\end{lemma}

The upshot of most ramification being small (Lemma~\ref{L:most-ramification-small}) is that it allows us to apply our ``trivial bounds'' on the $p$-adic logarithm (Lemma~\ref{L:log-bounds2}) at all but finitely many places. 
The sum over $p$ in the proof of explicit Szpiro can be broken down into three cases as shown in Figure~\ref{F:ramification}
\begin{figure}[h]\label{F:ramification}
	\begin{center}
		\includegraphics[scale=0.75]{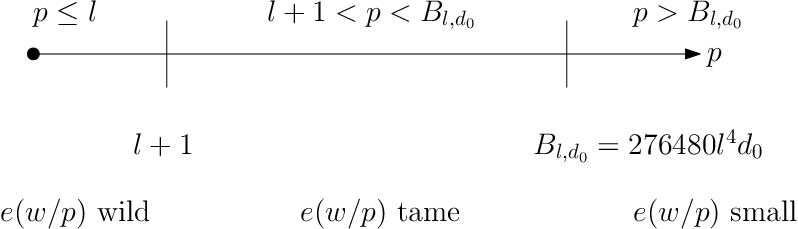}
	\end{center}
	\caption{A breakdown of the ramification of a tuple $\vec{\vu} = (\vu_0,\vu_1,\ldots,\vu_j) \in \Vu_p^{j+1}$. }
\end{figure}

%-----------------------------
\subsection{Explicit Szpiro}
%-----------------------------
In the remainder of the paper we derive the following version of Szpiro's inequality from \eqref{E:estimated}.

\begin{theorem}\label{T:szpiro}
Assume \cite[Corollary 3.12]{IUT3} and Claim \ref{L:arch-log-bounds}.
If $E/F$ is an elliptic curve in initial theta data  $(\Fbar/F, E_F, l, \Mu, \Vu, V^{\bad}_{\mod}, \epsu)$ built from the field of moduli then 
\begin{equation}\label{E:explicit-szpiro}
\vert \Delta^{\min}_{E/F}\vert \leq e^{A_0d_0^2l^4 + B_0d_0}( \vert \Cond(E/F) \vert \cdot \vert \Disc(F/\QQ) \vert)^{24+\varepsilon_l},
\end{equation}
where $A_0 = 84372107405, B_0 = 316495$ and $\varepsilon_l = (96 \left(l + 3\right))/(l^{2} + l - 12)$. %\frac{288}{l-11}$.
\end{theorem}
	Let $B = B_{l,d_0} = 2d_0\#\GL_2(\Z/30l)$. In what follows let 
$\E_p^2$ be the expected value of 
\begin{equation}\label{E:stuff-to-average}
\logmubar_{\vec{\vu}}(\overline{\Ical_{\vec{\vu}}}) + \Vert \diff_{\vec{\vu}} \Vert_1 - \Vert \diff_{\vec{\vu}} \Vert_\infty + 1_{\ram}(\vec{\vu})
\end{equation}
over $\vec{v} = (v_0,\ldots,v_j) \in \coprod_{j=1}^{(l-1)/2} V(F_0)_p^{j+1}$.
Above, $\overline{\Ical_{\vec{\vu}}}$ denotes the hull of the tensor product of log-shells for $\vec{\vu} \in \coprod_{j=1}^{(l-1)/2} \Vu_p^{j+1}$.
We compute $\EE_p^2$ for a given by breaking $p$ into the cases
\begin{itemize}
	\item infinite: $p=\infty$
	\item large: $p>B$
	\item small: $p\leq B$
\end{itemize} 
Also, within each case we break \eqref{E:stuff-to-average} into three subcomputations:
$$ \underbrace{\logmubar_{\vu}(\overline{\Ical_{\vu}})}_{\Irm} + \underbrace{\Vert \diff_{\vec{\vu}} \Vert_1 - \Vert \diff_{\vec{\vu}} \Vert_\infty}_{\IIrm} + \underbrace{1_{\ram}(\vec{\vu})}_{\IIIrm}.$$
We then put these estimates together to get our results. 

%------------------------------------------
\subsection{Computation at Infinite Places}
%------------------------------------------
 Over the infinite places we have 
	\begin{equation}\label{E:infinite-estimate}
	\E_{\infty}^2 \leq \frac{l+5}{4}\ln(\pi).
	\end{equation}

\begin{proof}
	At the infinite prime $\IIrm_{\infty}=\IIIrm_{\infty}=0$. 
	The number $\ln(\pi) (l+5)/4$ is just the uniform average of $(j+1)\ln(\pi)$ over $j$ which comes from Lemma~\ref{L:arch-log-bounds}.
\end{proof}
	
%------------------------------------------
\subsection{Computation at Large Places}
%------------------------------------------
Over the large places we have 
\begin{equation}\label{E:large-estimate}
\sum_{p>B, p\neq \infty} \E_p^2 \leq \frac{l+5}{4}\sum_{p>B, p \mid \vert D_{K,\Q}\vert }  \ln(p),
\end{equation}
which we can further estimate using 
$$\sum_{p>B, p \vert \vert \Disc(K/\QQ) \vert}\ln(p) \leq  2\left( \frac{\ln \vert \Disc(F/\QQ) \vert}{[F:\Q]} + \frac{\ln \vert \Cond(E/F) \vert }{[F:\Q]} \right).$$ 
We give a proof of these two claims.

\begin{proof}
	By the results of \S\ref{S:ramification-analysis} we know that for $p>B_{l,d_0}$, and every place $\vu \in \Vu_p$ we have 
	$$ e(\vu/p) < p-1.$$
	This leads to improvements in both the log-shell bounds $\Irm_p$ and the different bounds $\IIrm_p$.
	From the estimates on log-shells we know that for $\vec{\vu} \in \Vu_p^{j+1}$ that 
	$$ \overline{\Ical_{\vec{\vu}} }\subset D_{\vec{\vu}}(0; p^{j+1- \sum_{i=0}^j1/e(\vu/p)} ) \subset D_{\vec{\vu}}(0,p^{j+1}).$$
	This implies
	$$\logmubar_{\vec{\vu}}(\overline{\Ical_{\vec{\vu}}})\leq (j+1)\ln(p) $$
	and hence 
	$$ \E_p^2( \logmubar_{\vu}(\overline{\Ical_{\vu}})) \leq \frac{l+5}{4}\ln(p) 1_{\ram}(p). $$
	For the different term $\IIrm_p$ we have 
	$$ \E_p^2( \Vert \diff_{\vec{\vu}} \Vert_1 - \Vert \diff_{\vec{\vu}} \Vert_{\infty} \Vert )) \leq \frac{l+1}{4} \overline{\diff}_p \ln(p) $$
	where $\overline{\diff}_p = \log_p( \E( p^{\diff(\vu/p)}: \vu\in \Vu_p))$. 
	Using tameness, we have that $\diff(\vu/p) = 1-1/e(\vu/p)\leq 1$ which implies $\E(p^{\diff(\vu/p)}) \leq \E(p) = p$. 
	This gives 
	$$ \overline{\diff}_p \leq 1_{\ram}(p) 
	= \begin{cases}
	1, & \forall \vu \vert p \ \ e(\vu/p) =1, \\
	0, & \exists \vu \vert p, \ \ e(\vu/p) >1. 
	\end{cases}
	$$
	Hence 
	$$ \E_p^2\left( \Vert \diff_{\vec{\vu}} \Vert_1 - \Vert \diff_{\vec{\vu}} \Vert_{\infty} \right) \leq \frac{l+1}{4}1_{\ram}(p)\ln(p). $$
	Finally we estimate the third term $\IIIrm_p$:
	$$ \E_p^2(1_{\ram(\vec{\vu})}) \leq (1-\P_{\unr,p}^{\frac{l+1}{4}}) \leq 1_{\ram}(p). $$
	Putting the estimates for $\Irm_p$, $\IIrm_p$, and $\IIIrm_p$ together in the case that $p>B$ we get
	\begin{align*}
	\E_p^2 &\leq \frac{l+5}{4}\ln(p)1_{\ram}(p) + \frac{l+1}{4}\ln(p) 1_{\ram}(p) + 1_{\ram}(p) \\
	&\leq \left( \frac{l+5}{4} + \frac{l+5}{4}\right)\ln(p) 1_{\ram}(p) \\
	&= \frac{l+5}{2}\ln(p) 1_{\ram}(p).
	\end{align*}
	To finish our result we use the Lemma just outside this proof environment.\footnote{We have decided to label this theorem because it is a critical juncture where discriminants for $K$ meet conductors using N\'eron-Ogg-Shafarevich.
	This seems to be the critical step in relating the two.  }
\end{proof}

\begin{lemma}
	\begin{equation}
	\sum_{p \mid \vert \Disc(K/\QQ)\vert, p>B} \ln(p) \leq 2 \left ( \frac{\ln \vert \Disc(F/\QQ) \vert }{[F:\QQ]} + \frac{\ln \vert \Cond(E/F) \vert }{[F:\QQ]} \right) 
	\end{equation}
\end{lemma}
\begin{proof}
	The hard part of this formula is not getting too greedy, it seems.
	For $p>B$ we know that 
	 $$ p \mid \vert \Disc(K/\QQ) \vert \iff p \mid \vert \Disc(F/\QQ) \vert \mbox{ or } p \mid \vert \Cond(E/F) \vert. $$
	It is enough to show for each prime $p$ with $p>B$ and $p \mid \vert \Disc(F/\QQ) \vert$ we have 
	 $$ \ln(p) \leq 2 \left( \frac{-\ln \vert \Disc(F/\QQ)\vert_p  -\ln \vert \Cond(E/F) \vert_p }{[F:\QQ]}\right).$$
	Note that we are using $p$-adic absolute values to take the $p$-parts of these integers. 
	We observe that
	$$ - \ln \vert \Disc(F/\QQ)\vert_p = \sum_{w \in V(F)_p} f(w/p)d(w/p) \ln(p), $$
	$$ - \ln \vert \Cond(F/\QQ)\vert_p = \sum_{w\in V(F)_p} f(w/p)c_E(w) \ln(p) $$
	where we have used 
	 $$ \Cond(E/F) = \prod_w P_w^{c_E(w)}, \mbox{ \ \ \  } \Disc(F/\QQ) = \prod_w P_w^{d(w/p_w)}$$
	From \S\ref{S:diff-and-disc} we know that $d(w/p_w) = e(w/p_w)- 1$ since $p_w >B$. 
	Hence, it is enough to show that for each $p \mid \vert \Disc(K/\QQ) \vert$ that
	\begin{equation}\label{E:bigger-than-one}
	\frac{2 \sum_{w\vert p} \left( f(w/p)(e(w/p)-1) + f(w/p)c_E(w) \right) }{[F:\QQ]} \geq 1. 
	\end{equation}
	Using that $p>B$ and $2(e(w/p)-1) \geq e(w/p)$ together with the fact that $\sum_{w\in V(F)_p} f(w/p)e(w/p) = [F:\QQ]$ we get 
	\begin{align*}
	\mbox{ LHS of \eqref{E:bigger-than-one} } &\geq \frac{[F:\QQ] + \sum_{w\in V(F)_p} 2f(w/p)c_E(w) }{[F:\QQ]} \\
	&= 1+ 2\frac{\sum_{w\in V(F)_p}f(w/p) c_E(w) }{[F:\QQ]}.
	\end{align*}
	This proves the result. 
	We note that it is strictly greater than one since the initial theta data hypothesis says that there is a \emph{non-empty} set of primes in $V^{\bad}_{\mod}$ of bad reduction.
\end{proof}

%------------------------------------------
\subsection{Computation at Small Places}
%------------------------------------------

Over the small places we have\footnote{For $f(x)$ and $g(x)$ positive functions of a single real variable we write $f(x) \asymp g(x)$ as $x\to \infty$ if and only if $f(x) = O(g(x))$ and $g(x) = O(f(x))$ as $x\to \infty$. }
\begin{equation}\label{E:small-estimate}
\sum_{p \leq B} \E_p^2 \leq (l+3)\ln(B) \pi(B) \asymp l^5d_0.
\end{equation} 

\begin{proof}
	In the situation where $p\leq B_{l,d_0}$ we have worse bounds for $\E_p^2$. 
	We will not care so much about these bounds as they turn into the constant which appears in Szpiro's inequality.\footnote{On some level, of course, we do care because we would like better constants.
	This is secondary achieving \emph{some} Szpiro though. }
	
	In the first term $\Irm_p$ we use
	$$\Ical_{\vec{\vu}} \subset D_{\vec{\vu}}\left(0;p^{j+1} \prod_{i=0}^j \frac{e(\vu_i/p)}{\ln(p)\exp(1)} \right) \subset D_{\vec{\vu}}(0;p^{j+1}B^{j+1}).$$
	This gives 
	$$\logmubar_{\vec{\vu}}(\overline{\Ical_{\vec{\vu}}}) \leq \log_p(p^{j+1}B^{j+1})\ln(p) = (j+1) \ln(pB),$$
	which in turn gives (for $p$ ramified)
	$$\E_p^2( \logmubar_{\vec{\vu}}(\overline{\Ical_{\vec{\vu}}})) \leq \E( (j+1) \ln(pB)) = \frac{l+5}{4}\ln(pB) \leq \frac{l+5}{2}\ln(B), $$
	where the last inequality used $p\leq B$.
	
	For term $\IIrm_p$ involving differents, we have 
	$$\E_p^2( \Vert \diff_{\vec{\vu}} \Vert_1 - \Vert \diff_{\vec{\vu}} \Vert_{\infty} ) = \frac{l+1}{4} \overline{\diff}_p \ln(p).$$ 
	Since $\diff(\vu/p) \leq 1 - 1/e(\vu/p) + \ord_p(e(\vu/p))$ we get $\diff(\vu/p) \leq 1+ \ord_p[K:\Q]$ which proves 
	$$ \overline{\diff}_p \leq \log_p( \E(p^{1+\ord_p[K:\Q]})) = \log_p( p^{1+\ord_p[K:\Q]})=1+\ord_p([K:\Q]).$$
	Hence we have
	$$ \E_p^2( \Vert \diff_{\vec{\vu}} \Vert_1 - \Vert \diff_{\vec{\vu}} \Vert_{\infty}  ) \leq \frac{l+1}{4} \left( 1+ \ord_p[K:\Q]) \right)\ln(p) \leq \frac{l+1}{2}\ln(B) $$
	Finally in term $\IIIrm_p$ we have 
	$$ \E_p^2(1_{\ram}(\vec{\vu})) \leq (1-\P_{\unr,p}^{\frac{l+1}{4}}) \leq 1_{\ram}(p).$$
	Putting the estimates for $\Irm_p$, $\IIrm_p$, and $\IIIrm_p$ together we get
	\begin{align*}
	\sum_{p\leq B} \E_p^2 &\leq \sum_{p\leq B} \left[ \frac{l+5}{2}\ln(B)+ \frac{l+1}{2}\ln(B) + 1 \right]\\
	&\leq ((l+3)\ln(B) +1)\pi(B) \\
	&\leq (l+3)\ln(B)\pi(B).
	\end{align*}
	
	This gives our main result. 
	The asymptotic is then derived by using bounds in the prime number function $\pi(x)$.
	One such bound is Dusart's bound \cite{Dusart2010} which states that for $x>1$ we have 
	\begin{equation}\label{E:Dusart}
	\pi(x) \leq \frac{x}{\ln(x)} \left( 1+ \frac{1.3}{\ln(x)} \right).
	\end{equation}
	This then shows, using $B = 276480 l^4d_0$, that  
	 $$ (l+3)\ln(B)\pi(B) \leq (l+3) B \left(1+ \frac{1.3}{\ln(B)} \right) \asymp l^5 d_0 \mbox{ \ \ \ \ as } l\to \infty. $$
\end{proof}

\begin{remark}
	Using a slightly better form of Dusart's bound gives an $1/\ln(B)^2$ correction term. 
\end{remark}

%------------------------------------------
\subsection{Proof of Explicit Szpiro}
%------------------------------------------

	Working from 
	 \begin{equation}\label{E:reduction}
	 \deghatu_{\lgp}(P_{\Theta}) -\deghatu(P_q) \leq \sum_p \E_p^2  
	 \end{equation}
	The left hand side of \eqref{E:reduction} becomes
	  $$ \left( \frac{l(l+1)}{12} -1 \right)\left(\frac{1}{2l} \right) \frac{\ln \vert \Delta^{\min}_{E/F} \vert } {[F:\Q]},$$
	 and the right hand side of \eqref{E:reduction} becomes 
	 \begin{align*}
	 \sum_p \E_p^2 =& \sum_{p\leq B } \E_p^2 + \sum_{p>B} \E_p^2 + \E_{\infty}^2\\
	 \leq&(l+3)\ln(B)\pi(B) + \frac{l+5}{2} \cdot  2 \left ( \frac{\ln \vert \Disc(F/\QQ) \vert }{[F:\QQ]} + \frac{\ln \vert \Cond(E/F) \vert }{[F:\QQ]} \right) \\
	 & + \left( \frac{l+5}{4} \right)\ln(\pi)
	 \end{align*} 
We now divide both sides by $(l+5)$. 
The coefficient of the left hand side becomes 
 $$\left( \frac{l(l+1)}{12} -1 \right)\left(\frac{1}{2l} \right) \left(\frac{1}{l+5}\right) = \frac{l^2+l-12}{24l(l+5)} =: \frac{1}{24+\varepsilon_l},$$
where solving for $\varepsilon_l$ gives 
$$\varepsilon_l = \frac{96 \, {\left(l + 3\right)}}{l^{2} + l - 12}.$$
We now have 
\begin{equation}\label{E:pre-result}
\frac{1}{24+\varepsilon_l} \ln \vert \Delta^{\min}_{E/F} \vert \leq \left[ \ln(B)\pi(B) + \ln(\pi) \right][F:\QQ] + \ln \vert \Disc(F/\QQ)\vert + \ln \vert \Cond(E/F) \vert 
\end{equation}

Finally, using $d_1=276480$ (the upper bound on $[F:F_0]$) so that $B=l^4d_1d_0$, we get
\begin{align*}
\left[ \ln(B)\pi(B) + \ln(\pi) \right][F:\QQ] &\leq l^4d_0 \left( d_1\left( 1+\frac{1.3}{\ln(d_1)}\right) + \ln(\pi) \right) d_1d_0\\
&\leq A_0d_0^2l^4 + B_0 d_0 
\end{align*}
where $A_0 = 84372107405$, and $B_0 = 316495$.
This gives our result after rewriting \eqref{E:pre-result} multiplicatively with new bounds.

\bibliographystyle{amsalpha}
\bibliography{IUT.bib}

%\printbibliography

\end{document}